\documentclass[12pt,a4paper]{article}

\usepackage{amssymb}
\usepackage{amsmath, amsthm}
\usepackage{arydshln}
\usepackage{epsfig}
\usepackage{setspace}

\usepackage{geometry}

\def\RR{{\mathbb{R}}}
\def\ZZ{{\mathbb{Z}}}
\def\QQ{{\mathbb{Q}}}

\def\KK{{\mathbb{K}}}

\numberwithin{equation}{section}

\newcommand{\rank}{\mathop{\rm rank} }
\newcommand{\diam}{\mathop{\rm diam} }

\newcommand{\argmin}{\mathop{\rm argmin} }
\newcommand{\Det}{\mathop{\rm Det} }

\newcommand{\ncrank}{\mathop{\rm nc\mbox{-}rank} }

\newtheorem{Thm}{Theorem}[section]
\newtheorem{Prop}[Thm]{Proposition}
\newtheorem{Lem}[Thm]{Lemma}
\newtheorem{Cor}[Thm]{Corollary}
\theoremstyle{definition}

\newtheorem{Rem}[Thm]{Remark}

\title{Computing the nc-rank via discrete convex optimization on CAT(0) spaces}
\author{Masaki HAMADA and Hiroshi HIRAI \\
Department of Mathematical Informatics, \\
Graduate School of Information Science and Technology,   \\
The University of Tokyo, Tokyo, 113-8656, Japan.\\
\texttt{\normalsize masaki$\_$hamada@mist.i.u-tokyo.ac.jp}\\
\texttt{\normalsize hirai@mist.i.u-tokyo.ac.jp}
}
\begin{document}
\maketitle
\begin{abstract}
In this paper, we address the noncommutative rank (nc-rank) computation of
a linear symbolic matrix
\[
A = A_1 x_1 + A_2 x_2 + \cdots + A_m x_m,
\]
where each $A_i$ is an $n \times n$ matrix 
over a field $\KK$, and $x_i$ $(i=1,2,\ldots,m)$ are noncommutative variables.
For this problem, polynomial time algorithms were given 
by Garg, Gurvits, Oliveira, and Wigderson for $\KK = \QQ$, and by Ivanyos, Qiao, and Subrahmanyam for an arbitrary field $\KK$.
We present a significantly different polynomial time algorithm 
that works on an arbitrary field $\KK$.
Our algorithm is based on a combination of submodular optimization on modular lattices and 
convex optimization on CAT(0) spaces.
\end{abstract}


\section{Introduction}
The present article addresses rank computation 
of a {\em linear symbolic matrix}---
a matrix of the following form: 
\begin{equation}\label{eqn:A}
A = A_1 x_1 + A_2 x_2 + \cdots + A_m x_m,
\end{equation}
where each $A_i$ is an $n \times n$ matrix 
over a field $\KK$, $x_i$ $(i=1,2,\ldots,m)$ are variables, 
and $A$ is viewed as a matrix over $\KK(x_1,x_2,\ldots,x_m)$. 
This problem, sometimes called {\em Edmonds' problem},  
has fundamental importance in a wide range 
of applied mathematics and computer science; see~\cite{Lovasz89}. 
Edmonds' problem (on large field $\KK$) is a representative problem
that belongs to RP---the class of problems having a randomized polynomial time algorithm---but is not known to belong to P.
The existence of a deterministic polynomial time algorithm 
for Edmonds' problem is one of the major open problems in theoretical computer science.

In 2015, 
Ivanyos, Qiao, and Subrahmanyam~\cite{IQS15a} 
introduced a noncommutative formulation of the Edmonds' problem, called 
the {\em noncommutative Edmonds' problem}.
In this formulation, linear symbolic matrix $A$ is regarded as 
a matrix over the {\em free skew field} $\KK(\langle x_1,\ldots,x_m \rangle)$, which is the ``most generic" skew field of fractions of noncommutative polynomial ring 
$\KK \langle x_1,\ldots,x_m \rangle$. 
The rank of $A$ over the free skew field is called 
the {\em noncommutative rank}, or {\em nc-rank}, which is denoted by $\ncrank A$.
Contrary to the commutative case, 
the noncommutative Edmonds' problem 
can be solved in polynomial time.
\begin{Thm}[\cite{GGOW15,IQS15b}]
	The nc-rank of a matrix $A$ of form (\ref{eqn:A}) 
	can be computed in polynomial time.
\end{Thm}
As well as the result, the algorithms for nc-rank are stimulating
subsequent researches. 
The first polynomial time algorithm is due to 
Garg, Gurvits, Oliveira, and Wigderson~\cite{GGOW15} for the case of 
$\KK = \QQ$.
They showed that Gurvits' {\em operator scaling algorithm}~\cite{Gurvits04}, 
which was designed for solving a special class ({\em Edmonds-Rado class}) of Edmonds' problem, can solve nc-singularity testing (i.e., testing whether $n = \ncrank A$)
in polynomial time. 
The operator scaling algorithm has rich connections 
to various fields of mathematical sciences.
Particularly, nc-singularity testing can 
be formulated as a 
geodesically-convex optimization problem on Riemannian manifold $GL_n(\RR)/O_n(\RR)$, 
and the operator scaling can be viewed as 
a minimization algorithm on it; see \cite{AGLOW}.
For explosive developments after \cite{GGOW15},
we refer to e.g., \cite{BFGOWW} and the references therein. 

Ivanyos, Qiao, and Subrahmanyam~\cite{IQS15a,IQS15b} 
developed the first polynomial time algorithm for the nc-rank 
that works on an arbitrary field $\KK$. 
Their algorithm is viewed as a ``vector-space 
generalization" of the augmenting path algorithm 
in the bipartite matching problem.
This indicates a new direction 
in combinatorial optimization, since 
Edmonds' problem generalizes 
several important combinatorial optimization problems.
Inspired by their algorithm, 
\cite{HI_2x2} developed a combinatorial polynomial time algorithm
for a certain algebraically constraint 2-matching problem in a bipartite graph, which corresponds to the (commutative) Edmonds' problem for a linear symbolic matrix in~\cite{IwataMurota95}.
Also,  
a noncommutative algebraic formulation that 
captures weighted versions
of combinatorial optimization problems
was studied in \cite{HH_degdet,HI_degdet,Oki}.

The main contribution of this paper is 
a significantly different polynomial time algorithm 
for computing the nc-rank on an arbitrary field $\KK$. 
While describing 
the above algorithms and validity proofs is rather tough work,  
the algorithm and proof presented in this paper 
are conceptually simple, elementary, 
and relatively short. 
Further, it is also relevant to
the following two cutting edge issues in 
discrete and continuous optimization:
\begin{itemize}
	\item submodular optimization on a modular lattice.
	\item convex optimization on a CAT(0) space.
\end{itemize}

 A {\em submodular function} $f$ on a lattice ${\cal L}$
is a function $f:{\cal L} \to \RR$ 
satisfying $f(p) + f(q) \geq f(p \vee q) + f(p \wedge q)$ for $p,q \in {\cal L}$.
Submodular functions on Boolean lattice $\{0,1\}^n$ are well-studied, and
have played central roles in the developments of combinatorial optimization; see \cite{FujiBook}.
They are correspondents of convex functions 
({\em discrete convex functions}) 
in discrete optimization; see \cite{MurotaBook}.
Optimization of submodular functions beyond Boolean lattices, particularly on modular lattices, 
is a new research area that has just started; 
see \cite{FKMTT14,HH16L-convex,Kuivinen11} on this subject.

A {\em CAT(0) space} is a (non-manifold) generalization 
of nonpositively curved Riemannian manifolds; see~\cite{BrHa}. 
While CAT(0) spaces have been studied mainly in geometric group theory,
their effective utilization in applied mathematics 
has gained attention; see e.g.,\cite{BHV01}.
A CAT(0) space is a uniquely-geodesic metric space, 
and convexity concepts are defined along unique geodesics.
Theory of algorithms and optimization on CAT(0) spaces is now being pioneered; see e.g., ~\cite{Bacak13,Bacak14,BacakBook,Hayashi,Owen11}.

Our algorithm is obtained as a combination 
of these new optimization approaches.
We hope that this will bring new interactions to the nc-rank literature.
While it is somehow relevant to 
geodesically-convex optimization mentioned above,
we deal with optimization on 
combinatorially-defined non-manifold CAT(0) spaces.
The most important implication of our result is that   
convex optimization algorithms on such spaces can be a tool of showing polynomial time complexity.

\paragraph{Outline.} Let us outline our algorithm.
As shown by Fortin and Reutenauer~\cite{FortinReutenauer04}, 
the nc-rank is given by the optimum value of
an optimization problem:
\begin{Thm}[\cite{FortinReutenauer04}]\label{thm:FR}
	Let $A$ be a matrix of form~{\rm (\ref{eqn:A})}.
	Then $\ncrank A$ is equal to the optimal value of the following problem:
	\begin{eqnarray*}
		{\rm FR}: \quad {\rm Min.} && 2n - r - s  \\
		{\rm s.t.} && \mbox{$S A T$ has an $r \times s$ zero submatrix,} \\
		&& S, T \in GL_n(\KK).
	\end{eqnarray*}
\end{Thm}
As in~\cite{IQS15a,IQS15b}, our algorithm 
is designed to solve this optimization problem. 
This problem FR can also be formulated as an optimization problem on the modular lattice of vector subspaces in $\KK^n$,  as follows. 
Regard each matrix $A_i$ as a bilinear form $\KK^n \times \KK^n \to \KK$ by
\[
A_i(x,y) := x^{\top} A_i y \quad (x,y \in \KK^n).
\]
Then the condition of FR says that 
there is a pair of vector subspaces $U$ and $V$ of dimension $r$ and $s$, respectively, that
annihilates all bilinear forms, i.e., $A_i(U,V) := \{0\}$.
The objective function is written as $2n - \dim U - \dim V$.
Therefore, FR is equivalent to the following problem 
({\em maximum vanishing subspace problem; MVSP}):
\begin{eqnarray*}
	{\rm MVSP}: \quad
	{\rm Min.} &&  
	- \dim X - \dim Y \\
	{\rm s.t.} &&   A_i(X,Y) =\{0\} \quad (i=1,2,\ldots,m),  \\
	&& X,Y: \mbox{vector subspaces of $\KK^n$}.
\end{eqnarray*}
It is a basic fact that
the family ${\cal L}$ of all vector subspaces in $\mathbb{K}^n$
forms a modular lattice with respect to 
the inclusion order.
Hence, MVSP is an optimization problem over ${\cal L} \times {\cal L}$. 
Further, by reversing the order of the second ${\cal L}$, 
it can be viewed as a {\em submodular function minimization (SFM)} on modular lattice ${\cal L} \times {\cal L}$;
see Proposition~\ref{prop:submodular} in Section~\ref{subsec:submodular}.

Contrary to the Boolean case, 
it is not known generally whether a submodular function 
on a modular lattice can be minimized in polynomial time.
The reason of polynomial-time solvability of SFM on Boolean lattice $\{0,1\}^n$ 
is the {\em Lov\'asz extension}~\cite{Lovasz83}---
a piecewise-linear interpolation $\bar f: [0,1]^n \to \RR$ of function $f:\{0,1\}^n \to \RR$ such that 
$\bar f$ is convex if and only if $f$ is submodular.  
For SFM on a modular lattice, however, 
such a good convex relaxation to $\mathbb{R}^n$ is not known.

A recent study~\cite{HH16L-convex} introduced
an approach of constructing 
a convex relaxation of SFM on a modular lattice, 
where the domain of the relaxation is a CAT(0) space.
The construction is based on the concept of an {\em orthoscheme complex}~\cite{BM10}. 
Consider the order complex $K({\cal L})$ of ${\cal L}$,
and endow each simplex with a specific Euclidean metric.
The resulting metric space $K({\cal L})$ is called
the orthoscheme complex of ${\cal L}$, and 
is dealt with as a continuous relaxation of ${\cal L}$. 
The details are given in Section~\ref{subsub:K(L)}.
Figure~\ref{fig:folder} illustrates
the orthoscheme complex of a modular lattice with rank $2$, which is obtained by gluing Euclidean 
isosceles right triangles 
along longer edges.
	\begin{figure}[t]
	\begin{center}
		\includegraphics[scale=0.4]{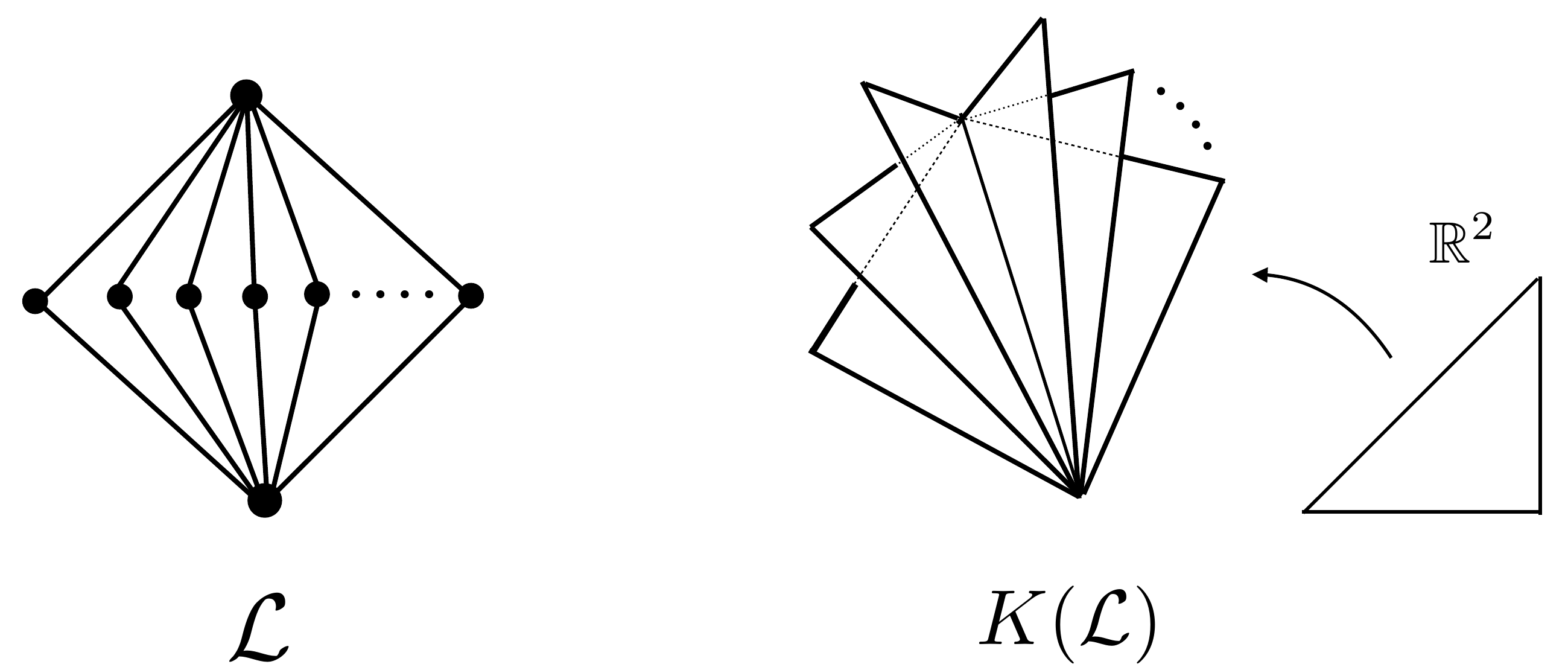}
		\caption{An orthoscheme complex}
		\label{fig:folder}
	\end{center}
\end{figure}\noindent
The orthoscheme complex of a modular lattice 
was shown to be CAT(0)~\cite{CCHO}. 
This enables us to consider geodesically-convexity 
for functions on $K({\cal L})$.
In this setting, a submodular function $f:{\cal L} \to \RR$ 
is characterized by the convexity of 
its piecewise linear interpolation, i.e., Lov\'asz extension 
$\bar f:K({\cal L}) \to \RR$~\cite{HH16L-convex}.
According to this construction, we obtain 
an exact convex relaxation of MVSP in a CAT(0)-space.

Our proposed algorithm is obtained by applying  
the {\em splitting proximal point algorithm (SPPA)} to 
this convex relaxation.
SPPA is a generic algorithm that minimizes a convex function of a separable form 
$\sum_{i=1}^N f_i$, 
where each $f_i$ is a convex function. 
Each iteration of the algorithm updates 
the current point $x$ to its {\em resolvent} of $f_i$---
a minimizer of 
$y \mapsto f_i(y) + (1/\lambda)d(y,x)^2$, where $i$ is chosen cyclically.
Ba\v{c}\'ak~\cite{Bacak14} showed that 
SPPA generates a sequence convergent to a minimizer of $f$
(under a mild assumption).
Subsequently, Ohta and P\'alfia~\cite{OhtaPalfia15} 
proved a sublinear convergence of SPPA.

The main technical contribution is to show that 
SPPA is applicable to the convex relaxation of MVSP 
and becomes a polynomial time algorithm for MVSP: 
We provide an equivalent convex relaxation of MVSP with
a separable objective function $\sum_i f_i$, and show that 
the resolvent of each $f_i$ can be computed in polynomial time. 
By utilizing the sublinear convergence estimate, 
a polynomial number of iterations for SPPA
identifies an optimal solution of MVSP.

Compared with the existing algorithms, 
this algorithm has advantages and drawbacks.  
As mentioned above, 
our algorithm and its validity proof are relatively simple. 
Particularly, it can be uniformly written for an arbitrary field $\KK$, where
only the requirement for $\KK$ is that arithmetic operations is executable. 
No care is needed for a small finite field, 
whereas the algorithm in \cite{IQS15a,IQS15b} needs a field extension.
On the other hand, our algorithm is very slow; see Theorem~\ref{thm:prox}.
This is caused by using a generic and primitive algorithm (SPPA) 
for optimization on CAT(0) spaces. 
We believe that this will be naturally improved in future developments. 

The problematic point of our algorithm 
is bit-complexity explosion for the case of $\KK = \QQ$.
Our algorithm updates feasible vector subspaces in MVSP, 
and can cause an exponential increase of the bit-size 
representing bases of those vector subspaces.
To resolve this problem and
make use of the advantage in finite fields, 
we propose a reduction of nc-rank computation on $\QQ$ 
to that on $GF(p)$.
This reduction is an application of the $p$-adic valuation 
on $\QQ$.
We consider a weighted version of the nc-rank, which 
was introduced by~\cite{HH_degdet} for $\KK(t)$
 and is definable 
for an arbitrary field with a discrete valuation.
The corresponding optimization problem MVMP is  
a discrete convex optimization on a representative CAT(0) space---the {\em Euclidean building} for $GL_n(\QQ)$ (or $GL_n(\QQ_p)$). 
This may be viewed as a 
$p$-adic counterpart of the above geodesically-convex optimization approach on $GL_n(\RR)/O_n(\RR)$ 
for nc-singularity testing on $\QQ$. 
By using an obvious relation of the $p$-adic valuation of 
a nonzero integer and its bit-length in base $p$,  
we show that nc-singularity testing on $\QQ$ reduces to a polynomial number of nc-rank computation over the residue field $GF(p)$,  
in which the required bit-length is polynomially bounded.

\paragraph{Organization.}
The rest of this paper is organized as follows.
In Section~\ref{sec:pre}, 
we present necessary backgrounds on convex optimization on CAT(0) space,
modular lattices, and submodular functions. 
In Section~\ref{sec:algorithm}, 
we present our algorithm and show its validity.
In Section~\ref{sec:p-adic}, 
we present the $p$-adic reduction for
nc-rank computation on $\QQ$.

\paragraph{Original motivation: 
Block triangularization of a partitioned matrix.}
The original version~\cite{HamadaHirai} of this paper  dealt with block triangularization of a matrix with 
the following partition structure:
 \[
 A = \left(
 \begin{array}{ccccc}
 A_{11} & A_{12} &\cdots & A_{1\nu} \\
 A_{21} & A_{22} &\cdots & A_{2\nu} \\
 \vdots & \vdots & \ddots & \vdots \\
 A_{\mu1}&A_{\mu2} &\cdots & A_{\mu \nu}
 \end{array}\right),
 \]
 where $A_{\alpha \beta}$ is an $n_\alpha \times m_\beta$ 
 matrix over field $\KK$ for $\alpha \in [\mu]$ and
 $\beta\in [\nu]$. 
 Consider the following block triangularization 
 \[
 A\mapsto PEAFQ
 =\left[\begin{array}{ccccc}
 \multicolumn{1}{|c}{} & & &  \\\cline{1-1}
 & \multicolumn{1}{|c}{} && \text{\huge $*$}\\\cline{2-3}
 & & & \multicolumn{1}{|c}{} &\\\cline{4-4}
 \text{\huge $O$} & & & & \multicolumn{1}{|c}{} \\\cline{5-5}
 \end{array}\right],
 \]
 where $P$ and $Q$ are permutation matrices and 
 $E$ and $F$ are regular transformations ``within blocks," 
 i.e., $E$ and $F$ are block diagonal 
 matrices with block diagonals 
 $E_\alpha \in GL_{n_{\alpha}}(\KK)$ $(\alpha \in [\mu])$ and $F_\beta \in GL_{m_{\beta}}(\KK)$ 
 $(\beta \in [\mu])$, respectively.
 Such a block triangularization was addressed by Ito, Iwata, and Murota~\cite{ItoIwataMurota94} for motivating 
 analysis on physical systems with (restricted) symmetry.
 The most effective block triangularization 
 is determined by arranging a maximal chain of 
 maximum-size zero-blocks exposed in $EAF$, 
 where the size of a zero block is defined as 
 the sum of row and column numbers.
This generalizes the classical {\em Dulmage-Mendelsohn decomposition} 
for bipartite graphs and Murota's {\em combinatorial canonical form} for layered mixed matrices; see \cite{HH16DM,MurotaBook}.

Finding a maximum-size zero-block 
is nothing but FR (or MVSP) for
the linear symbolic matrix obtained by multiplying variable $x_{\alpha\beta}$ to $A_{\alpha\beta}$; 
see \cite[Appendix]{HH_degdet} 
for details. 
The original version of our algorithm was 
designed for this zero-block finding. 
Later, we found that this is essentially nc-rank computation. 
This new version improves analysis (on Theorem~\ref{thm:prox}), 
simplifies the arguments, particularly 
the proof of Theorem~\ref{thm:P2}, and includes 
the new section for the $p$-adic reduction.

\section{Preliminaries}\label{sec:pre}

Let $[n]$ denote $\{1,2,\ldots,n\}$.
Let $\RR$, $\QQ$, $\ZZ$ denote the sets of real, rational, and integer numbers, respectively.
Let $1_X$ denote the vector in $\RR^n$ 
such that $(1_X)_i = 1$ if $i \in X$ and zero otherwise.
The $i$-unit vector $1_{\{i\}}$ is simply written as $1_i$.

\subsection{Convex optimization on CAT(0)-spaces}\label{sec:CAT(0)}

\subsubsection{CAT(0)-spaces}
Let $K$ be a metric space with distance function $d$. 
A {\em path} in $K$ is a continuous map $\gamma:[0,1] \to L$, where
its length is defined as $\sup \sum_{i=0}^{N-1} d(\gamma(t_i), \gamma(t_{i+1}))$
over $0=t_0 < t_1 < t_2 < \cdots < t_N = 1$ and $N > 0$.
If $\gamma(0) = x$ and $\gamma(1) = y$, then 
we say that a path $\gamma$ {\em connects} $x,y$. 
A {\em geodesic} is a path $\gamma$ satisfying
$d(\gamma(s), \gamma(t)) = d(\gamma(0), \gamma(1)) |s - t|$ 
for every $s,t \in [0,1]$.
A {\em geodesic metric space} is a metric space $K$ in which
any pair of two points is connected by a geodesic.
Additionally, 
if a geodesic connecting any points is unique, then
$K$ is called {\em uniquely geodesic}.

We next introduce a CAT(0) space.
Informally, it is defined as a geodesic metric space 
in which any triangle is not thicker than the corresponding triangle in Euclidean plane. 
We here adopt the following definition.
A geodesic metric space $K$ is said to be {\em CAT(0)} 
if for every point $x \in K$, every geodesic $\gamma:[0,1] \to K$ and $t \in [0,1]$, it holds
\begin{equation}\label{eqn:CAT(0)}
d(x,\gamma(t))^2 \leq (1-t) d(x,\gamma(0))^2 + td(x,\gamma(1))^2 - t(1-t)d(\gamma(0),\gamma(1))^2.
\end{equation}
The following property of a CAT(0) space is a basis of introducing convexity.
\begin{Prop}[{\cite[Proposition 1.4]{BrHa}}]\label{prop:uniquely-geodesic}
	A {\rm CAT}$(0)$-space is uniquely geodesic.
\end{Prop}

Suppose that $K$ is a {\rm CAT}$(0)$ space.
For points $x,y$ in $K$, let $[x,y]$ 
denote the image of a unique geodesic $\gamma$ connecting $x,y$.
For $t \in [0,1]$, the point $p$ on $[x,y]$ 
with $d(x,p)/d(x,y) = t$ is formally written 
as $(1-t)x + t y$.

A function $f: K \to \RR$ is said to be {\em convex}
if for all $x,y \in K, t \in [0,1]$
it satisfies 
\begin{equation*}
 f((1-t) x + t y) \leq (1-t) f(x) + t f(y).
\end{equation*}
If it satisfies a stronger inequality
\begin{equation*}
f((1-t) x + t y) \leq (1-t) f(x) + t f(y) - \frac{\kappa}{2} t(1-t) d(x,y)^2
\end{equation*}
for some $\kappa > 0$, then $f$ is said to be 
{\em strongly convex} with parameter $\kappa > 0$.
In this paper, 
we always assume that a convex function is continuous.
A function $f:K \to \RR$ is said to be {\em $L$-Lipschitz}
with parameter $L \geq 0$
if for all $x,y \in K$
it satisfies 
\begin{equation*}
	|f(x) - f(y)| \leq L d(x,y).
\end{equation*}

\begin{Lem}\label{lem:d^2}
	For any $z \in K$,
	the function $x \mapsto d(z,x)^2$ is strongly convex with parameter $\kappa = 2$, and 
	is $L$-Lipschitz with $L = 2\diam K$, where 
	$\diam K:= \sup_{x,y \in K} d(x,y)$ denotes the diameter of $K$
	
\end{Lem}
The former follows directly from the definition (\ref{eqn:CAT(0)}) of CAT(0)-space.
The latter follows from $d(z,x)^2 - d(z,y)^2 \leq 
(d(z,x)+ d(z,y))(d(z,x) - d(z,y)) 
= (d(z,x)+ d(z,y)) d(x,y) \leq (2 \diam K) d(x,y)$.

\subsubsection{Proximal point algorithm}
Let $K$ be a complete CAT(0)-space (which 
is also called an {\em Hadamard space}).
For a convex function $f: K \to \RR$ and $\lambda > 0$
the {\em resolvent} of $f$ is a map $J_\lambda^f: K \to K$ defined by
\begin{equation*}
J_{\lambda}^f (x) := \argmin_{y \in K} \left( f(y) + \frac{1}{2\lambda} d(x,y)^2 \right)
\quad (x \in K).
\end{equation*}
Since the function $y \mapsto f(y) + \frac{1}{2\lambda} d(x,y)^2$ is strongly convex with parameter $1/\lambda > 0$, 
the minimizer is uniquely determined, and $J_{\lambda}^f$ is well-defined; see \cite[Proposition 2.2.17]{BacakBook}.

The {\em proximal point algorithm (PPA)} is 
to iterate updates $x \leftarrow J_{\lambda}^f(x)$.
This simple algorithm generates a sequence converging to a minimizer of $f$
under a mild assumption; see \cite{Bacak13,BacakBook}.
The {\em splitting proximal point algorithm (SPPA)}~\cite{Bacak14,BacakBook}, which we will use, minimizes a convex function $f: K \to \RR$ represented as the following form
\begin{equation*}
	f := \sum_{i=1}^m f_i,
\end{equation*}
where each $f_i:K \to \RR$ is a convex function.
Consider a sequence $(\lambda_k)_{k=1,2,\ldots,}$ satisfying
\begin{equation*}
\sum_{k=0}^{\infty} \lambda_k = \infty, \quad \sum_{k=0}^\infty \lambda_k^2 < \infty.
\end{equation*}
 
\begin{description}
\item[Splitting Proximal Point Algorithm (SPPA)]
\item[$\bullet$] Let $x_0 \in K$ be an initial point.
\item[$\bullet$] For $k=0,1,2,\ldots$, repeat the following:
\[
x_{km+i} := J_{\lambda_k}^{f_i} (x_{km+i -1}) \quad (i=1,2,\ldots,m).
\]
\end{description}
Ba\v{c}\'{a}k~\cite{Bacak14} showed that the sequence generated by
SPPA converges to a minimizer of $f$ if 
$K$ is locally compact.
Ohta and P\'alfia~\cite{OhtaPalfia15} 
proved sublinear convergence of SPPA if 
$f$ is strongly convex and $K$ is not necessarily locally compact.
\begin{Thm}[{\cite{OhtaPalfia15}}]\label{thm:OhtaPalfia}
Suppose that 
$f$ is strongly convex with parameter $\epsilon > 0$
and each $f_i$ is $L$-Lipschitz.
Let $x^*$ be the unique minimizer of $f$.
Define the sequence $(\lambda_k)$ by
\begin{equation*}
\lambda_k := 1/ \epsilon (k+1).
\end{equation*} 
Then the sequence $(x_\ell)$ generated by SPPA satisfies
\begin{equation*}
d(x_{km}, x^*)^2 = O\left(\frac{\log k}{k} \frac{L^2m^2}{\epsilon^2} \right)
\quad (k =1,2,\ldots).
\end{equation*} 
\end{Thm}

\subsection{Geometry of modular lattices}\label{sec:modularlattice}
We use basic terminologies and facts in lattice theory; see e.g., \cite{Gratzer}.
A {\em lattice} ${\cal L}$ 
is a partially ordered set
in which every pair $p,q$ of elements 
has meet $p \wedge q$ (greatest common lower bound) 
and join $p \vee q$ (lowest common upper bound). 
Let $\preceq$ denote the partial order, where 
$p \prec q$ means $p \preceq q$ and $p \neq q$.
A pairwise comparable subset of ${\cal L}$, 
arranged as $p_0 \prec p_1 \prec \cdots \prec p_k$,
is called a {\em chain} (from $p_0$ to $p_k$),
where $k$ is called the length. 
In this paper, we only consider lattices in which 
any chain has a finite length.
Let ${\bf 0}$ and ${\bf 1}$ denote the minimum and maximum elements of ${\cal L}$, respectively.
The rank $r(p)$ of element $p$ is defined 
as the maximum length of a chain from ${\bf 0}$ to $p$. 
The rank of lattice ${\cal L}$ is defined as the rank of ${\bf 1}$. 
For elements $p,q$ with $p \preceq q$
the {\em interval} $[p, q]$ is the set of elements $u$ with $p \preceq u \preceq q$.
Restricting $\preceq$ to $[p, q]$, the interval  $[p, q]$ 
is a lattice with maximum $q$ and minimum $p$. 
If $p \neq q$ and $[p,q] = \{p,q\}$, we say that $q$ {\em covers} $p$ and denote $p \prec: q$ or $q :\succ p$.
For two lattices ${\cal L}, {\cal M}$, 
their direct product ${\cal L} \times {\cal M}$ 
becomes a lattice, where the partial order on ${\cal L} \times {\cal M}$ 
 is defined by $(p,p') \preceq (q,q') \Leftrightarrow p \preceq q, p' \preceq q'$.

A lattice ${\cal L}$ is called {\em modular} if 
for every triple $x,a,b$ of elements with $x \preceq b$, 
it holds $x \vee (a \wedge b) = (x \vee a) \wedge b$.
A modular lattice satisfies the Jordan-Dedekind chain condition. This is, the lengths of 
maximal chains of every interval 
are the same. 
Also, we often use the following property:
\begin{equation}\label{eqn:basic}
p \prec: p'\ \Rightarrow \ p \wedge q = p' \wedge q \ \
\mbox{or} \ \  p \wedge q \prec: p' \wedge q. 
\end{equation}
This can be seen from the definition of modular lattices, 
and holds also when replacing $\wedge$ by $\vee$.

A modular lattice ${\cal L}$ is said to be {\em complemented} 
if every element can be represented as a join of atoms, where
an {\em atom} is an element of rank $1$.
It is known that for a complemented modular lattice, every interval is complemented modular, 
and a lattice obtained by reversing the partial order 
is also complemented modular. 
The product of two complemented modular lattices is also complemented modular.

A canonical example of a complemented modular lattice is 
the family ${\cal L}$ of all subspaces of a vector space $U$, 
where the partial order is the inclusion order with 
$\wedge = \cap$, and $\vee = +$.
Another important example is a {\em Boolean lattice}---a lattice isomorphic to 
the poset $2^{[n]}$ of all subsets 
of $[n]$ with respect to the inclusion order $\subseteq$.

\subsubsection{Frames---Boolean sublattices in a complemented modular lattice}
Let ${\cal L}$ be a complemented modular lattice of rank $n$, and let
$r$ denote the rank function of ${\cal L}$. 
A complemented modular lattice is equivalent 
to a {\em spherical building of type A}~\cite{BuildingBook}. 
We consider a lattice-theoretic counterpart of an {\em apartment}, which is a maximal Boolean sublattice of ${\cal L}$.

A {\em base} is a set of $n$ atoms $a_1,a_2,\ldots,a_n$ 
with $a_1 \vee a_2 \vee \cdots \vee a_n = {\bf 1}$.
The sublattice $\langle a_1,a_2,\ldots, a_n \rangle$ 
generated by a base $\{a_1,a_2,\ldots,a_n\}$
is called a {\em frame}, which is isomorphic
to a Boolean lattice $2^{[n]}$ 
by the map
\[
X \mapsto \bigvee_{i \in X} a_i. 
\]
\begin{Lem}[{see e.g.,\cite{Gratzer}}]\label{lem:frame}
Let ${\cal L}$ be a complemented modular lattice of rank $n$.
\begin{itemize}
	\item[(1)] For chains ${\cal C},{\cal D}$ in $\cal L$,  
	there is a frame ${\cal F} \subseteq {\cal L}$ containing $\cal C$ and $\cal D$.
    \item[(2)] For a frame ${\cal F}$ and an ordering $a_1,a_2,\ldots, a_n$ of its basis, 
    define map $\varphi_{a_1,a_2,\ldots, a_n}: {\cal L} \to {\cal F}$ by
    \begin{equation}\label{eqn:retraction}
    p \mapsto \bigvee \{a_i \mid  i \in [n]: p \wedge (a_1 \vee a_2 \vee \cdots \vee a_{i}) :\succ  p \wedge (a_1 \vee a_2 \vee \cdots \vee a_{i-1}) \}.
    \end{equation}
    Then $\varphi_{a_1,a_2,\ldots, a_n}$ is a retraction to ${\cal F}$ such that it is rank-preserving (i.e., $r(p) = r(\varphi(p))$)
    and order-preserving (i.e., $p \preceq q \Rightarrow \varphi(p) \preceq \varphi(q)$).
\end{itemize}
\end{Lem}
This is nothing but a part of the axiom of building, 
where the map in (2) is essentially
a {\em canonical retraction} to an apartment.
\begin{proof}	
	We show (1) by the induction on $n$.
	Suppose that ${\cal C} = ({\bf 0} = p_0 \prec p_1 \prec \cdots \prec p_n = {\bf 1})$
	and ${\cal D} = ({\bf 0} = q_0 \prec q_1 \prec \cdots \prec q_n = {\bf 1})$.
	Consider the maximal chains ${\cal C}', {\cal D}'$ 
	from ${\bf 0}$ to $p_{n-1}$, 
	where  ${\cal C}' := ({\bf 0} = p_0 \prec p_1 \prec \cdots \prec p_{n-1})$ 
	and ${\cal D}'$ 
	consists of $q_i' := p_{n-1} \wedge q_i$ $(i=0,1,\ldots,n)$.
	Note that the maximality of ${\cal D}'$ follows from (\ref{eqn:basic}).
	By induction, there is a frame $\langle a_1,a_2,\ldots,a_{n-1} \rangle$ 
	of the interval $[{\bf 0}, p_{n-1}]$ (that is a complemented modular lattice of rank $n-1$)
	such that it contains ${\cal C}', {\cal D}'$. 
	Consider the first index $j$ such that $q_j \not \preceq p_{n-1}$.
	Then $q'_i = q_i$ for $i<j$, and $q_j'=q_{j-1}$. 
	For $i \geq j$, 
	by $p_{n-1} \vee q_j = {\bf 1}$ and modularity, it holds that 
	$q_i$ covers $q_i'$. Again by modularity, 
	it must hold $q_i' \vee q_j = q_i$ for $i \geq j$.
	By complementality, 
	we can choose an atom $a_n$ 
	such that $q_{j-1} \vee a_n = q_j$.
	Now $\langle a_1,a_2,\ldots, a_n \rangle$ is a frame as required.
	
		(2). By (\ref{eqn:basic}), $\{p \wedge (a_1 \vee \cdots \vee a_i) \}_i$
	is a maximal chain from ${\bf 0}$ to $p$.
	From this and the chain condition, 
	the rank-preserving property follows.
	Suppose that $p \preceq q$ and $p \wedge 
	b \prec: p \wedge b'$ for $b \prec: b'$. 
	Then $[p \wedge b, b] \ni q \wedge b \preceq  q \wedge b'  \in [p \wedge b, b']$. 
	By (\ref{eqn:basic}) and the chain condition 
	from $p \wedge b$ to $b'$, it must hold 
	$q \wedge b \prec: q \wedge b'$. 
	This means that any index $i$ 
	appeared in (\ref{eqn:retraction}) for $p$ also appears in that for $q$. Then, the order-preserving property follows.	
\end{proof} 
Suppose that ${\cal L}$ is the lattice of all vector subspaces of $\KK^n$, 
and that we are given two chains ${\cal C}$ and ${\cal D}$ of vector subspaces, 
where each subspace $X$ in the chains is given by a matrix $B$ 
with ${\rm Im}\, B = X$ (or $\ker B = X$). 
The above proof can be implemented by Gaussian elimination, 
and obtain vectors  $a_1,a_2,\ldots,a_n$ 
with ${\cal C}, {\cal D} \subseteq \langle a_1,a_2,\ldots, a_n \rangle$
in polynomial time.

\subsubsection{The orthoscheme complex of a modular lattice}\label{subsub:K(L)} 
Let ${\cal L}$ be a modular lattice of rank $n$.
Let $K({\cal L})$ denote the {\em geometric realization} 
of the {\em order complex} of ${\cal L}$. That is,
$K({\cal L})$ is the set of all formal convex combinations 
$x = \sum_{p \in {\cal L}} \lambda(p) p$  of elements in ${\cal L}$ such that
the {\em support} $\{p \in {\cal L} \mid \lambda(p) \neq 0 \}$ of $x$ 
is a chain of ${\cal L}$.
Here ``convex" means that the coefficients $\lambda(p)$ are nonnegative reals and 
$\sum_{p \in {\cal L}} \lambda(p) = 1$.
A {\em simplex} corresponding to a chain $\cal C$
is the subset of points whose supports belong to $\cal C$.

We next introduce a metric on $K({\cal L})$. 
For a maximal simplex $\sigma$ corresponding to 
a maximal chain ${\cal C} = p_0 \prec p_1 \prec \cdots \prec p_n$, 
define a map $\varphi_{\sigma}: \sigma \to \RR^n$ 
by
\begin{equation}\label{eqn:phi_sigma}
\varphi_{\sigma}(x) = \sum_{i=1}^{n} \lambda_i 1_{[i]} 
\quad (x = \sum_{i=0}^n \lambda_i p_i \in \sigma).
\end{equation}
This is a bijection from $\sigma$ to the $n$-dimensional simplex of vertices 
$0, 1_{[1]}, 1_{[2]}, 1_{[3]}, \ldots, 1_{[n]}$.
This simplex is called the $n$-dimensional {\em orthoscheme}.
The metric $d_{\sigma}$ 
on each simplex $\sigma$ of $K({\cal L})$ is defined by
\begin{equation}\label{eqn:dsigma}
d_{\sigma}(x,y) := \| \varphi_\sigma (x) - \varphi_\sigma (y) \|_2 \quad (x,y \in \sigma). 
\end{equation}
Accordingly, the length $d(\gamma)$ of a path $\gamma: [0,1] \to K({\cal L})$
is defined as the supremum of $\sum_{i=0}^{N-1} d_{\sigma_i}(\gamma(t_i),\gamma(t_{i+1}))$
over all
$0 = t_0 < t_1 < t_2 < \cdots < t_N = 1$ and $N \geq 1$,  
in which $\gamma([t_i,t_{i+1}])$ belongs to a simplex $\sigma_i$ for each $i$.
Then the metric $d(x,y)$ on $K({\cal L})$
is defined  as the infimum of $d(\gamma)$ over all 
paths $\gamma$ connecting $x,y$.
The resulting metric space $K({\cal L})$ 
is called the {\em orthoscheme complex} of ${\cal L}$~\cite{BM10}.
By Bridson's theorem~\cite[Theorem 7.19]{BrHa}, 
$K({\cal L})$ is a complete 
geodesic metric space.
Basic properties of the orthoscheme complex of a modular lattice are summarized as follows.
\begin{Prop}\label{prop:basic_K(L)}
	\begin{itemize}
		\item[(1)]  \cite{CCHO} For a modular lattice ${\cal L}$,
		the orthoscheme complex $K({\cal L})$ is 
		a complete CAT(0) space. 
		\item[(2)] \cite{BM10,CCHO} For two 
		modular lattices ${\cal L}, {\cal M}$, the  orthoscheme complex  $K({\cal L} \times {\cal M})$ is isometric to $K({\cal L}) \times K({\cal M})$ with metric given by
		\[
		d((x,y),(x',y')) := \sqrt{ d(x,x')^2  + d(y,y')^2} \quad ((x,y),(x',y') \in K({\cal L}) \times K({\cal M})).
		\] 
		\item[(3)]\cite{BM10,CCHO}
		For a Boolean lattice ${\cal L} = 2^{[n]}$,
		the orthoscheme complex $K({\cal L})$ is isometric 
		to the $n$-cube $[0,1]^n \subseteq \RR^n$, where the isometry is given by
		\begin{equation}\label{eqn:F-coordinate}
		x  = \sum_{i} \lambda_i X_i
		\mapsto  \sum_{i} \lambda_i 1_{X_i}.
		\end{equation}
		\item[(4)] \cite{CCHO}
For a complemented modular lattice ${\cal L}$ of rank $n$ and 
a frame ${\cal F}$ of ${\cal L}$ with an ordering $a_1,a_2,\ldots,a_n$ of its basis, 
the map $\varphi = \varphi_{a_1,a_2,\ldots, a_n}: {\cal L} \to {\cal F}$ is extended to 
$\bar \varphi:K({\cal L}) \to K({\cal F})$ by
\begin{equation*}
x= \sum_{i} \lambda_i p_i  \mapsto  \sum_{i} \lambda_i \varphi(p_i).
\end{equation*}
Then $\bar \varphi$ is a nonexpansive retraction from $K({\cal L})$ to $K({\cal F})$. 
In particular, 
\begin{itemize}
	\item[(4-1)]  $K({\cal F}) \simeq [0,1]^n$ is 
	an isometric subspace of $K({\cal L})$, and
	\item[(4-2)] $\diam K({\cal L}) = \sqrt{n}$.
\end{itemize}
 	\end{itemize}
\end{Prop}
For a complemented modular lattice ${\cal L}$, 
the CAT(0)-property of $K({\cal L})$ is equivalent to the CAT(1)-property 
of the corresponding spherical building, as shown in~\cite{HKS17}. 
\paragraph{The isometry between $K({\cal L}) \times K({\cal M})$ and $K({\cal L} \times {\cal M})$.}

The isometry from $K({\cal L} \times {\cal M})$ to 
$K({\cal L}) \times K({\cal M})$ (Proposition~\ref{prop:basic_K(L)}~(2))
is given by 
\begin{equation*}
z = \sum_{i} \lambda_i (p_i,q_i) \mapsto \left(\sum_{i} \lambda_i p_i, \sum_i \lambda_i q_i \right). 
\end{equation*}
The inverse map is constructed as follows: For $(x,y) = (\sum_{i} \mu_i p_i, \sum_j \nu_j q_j) =: (x',y')$,  
choose maximum $p_i, q_j$ with $\mu_i \neq 0$, $\nu_j \neq 0$, 
set $z \leftarrow z + \min (\mu_i,\nu_j) (p_i,q_j)$,  
$(x',y') \leftarrow (x',y') -  \min (\mu_i,\nu_j) (p_i,q_j)$, repeat it from $z=0$ until $(x',y') = (0,0)$.
The resulting $z$ satisfies $\varphi(z) = (x,y)$.

\paragraph{The ${\cal F}$-coordinate of a frame ${\cal F}$.}
A frame ${\cal F} = \langle a_1,a_2,\ldots,a_n \rangle$ 
is isomorphic to Boolean lattice $2^{[n]}$ 
by $a_{i_1} \vee a_{i_2} \vee \cdots \vee a_{i_k} \mapsto \{i_1,i_2,\ldots,i_k\}$. 
Further, the subcomplex $K({\cal F})$ is viewed as an $n$-cube $[0,1]^n$, and
a point $x$ in $K({\cal F})$ 
is viewed as $x = (x_1,x_2,\ldots,x_n) \in [0,1]^n$
via isometry $(\ref{eqn:F-coordinate})$.
This $n$-dimensional vector $(x_1,x_2,\ldots,x_n)$ is called 
the {\em ${\cal F}$-coordinate} of $x$. 
From ${\cal F}$-coordinate $(x_1,x_2,\ldots,x_n)$, 
the original expression of $x$ is recovered by
sorting $x_1,x_2,\ldots,x_n$ in decreasing order as: $x_{i_1} \geq x_{i_2} \geq \cdots \geq x_{i_n}$, and letting
\begin{equation}\label{eqn:recover}
x = (1- x_{i_1}){\bf 0} + \sum_{k=1}^n (x_{i_k} - x_{i_{k+1}}) (a_{i_1} \vee a_{i_2} \vee \cdots \vee a_{i_k}),
\end{equation}
where $x_{i_{n+1}} := 0$.

\subsubsection{Submodular functions and Lov\'asz extensions}

Let ${\cal L}$ be a modular lattice.
A function $f:{\cal L} \to \RR$ is said to be {\em submodular} if
\begin{equation*}
f(p) + f(q) \geq f(p \wedge q) + f(p \vee q) \quad (p,q \in {\cal L}).
\end{equation*}
For a function $f:{\cal L} \to \RR$, 
the {\em Lov\'asz extension} $\overline f: K({\cal L}) \to \RR$  is defined by
\begin{equation*}
\overline f (x) := \sum_{i} \lambda_i f(p_i) \quad  (x = \sum_{i} \lambda_i p_i \in K({\cal L}) ).
\end{equation*}
In the case of ${\cal L} = 2^{[n]}$, 
this definition of the Lov\'asz extension coincides with the original one~\cite{FujiBook,Lovasz83}
by $K({\cal L}) \simeq [0,1]^n$ (Proposition~\ref{prop:basic_K(L)}~(3)).
\begin{Prop}\label{prop:Lovasz_ext}
Let ${\cal L}$ be a modular lattice of rank $n$.
For a function $f: {\cal L} \to \RR$, we have the following.
\begin{itemize}
\item[(1)] \cite{HH16L-convex} $f$ is submodular if and only if 
the Lov\'asz extension  $\overline{f}$ is convex. 
\item[(2)] The Lov\'asz extension $\overline{f}$  
is $L$-Lipschitz with
$
L = 2 \sqrt{n} \max_{p \in {\cal L}} |f(p)|.
$
\item[(3)] Suppose that $f$ is integer-valued. 	
For $x \in K({\cal L})$,
if $\overline{f}(x) - \min_{p \in {\cal L}} f(p) < 1$, then
a minimizer of~$f$ exists in the support of $x$.
\end{itemize}
\end{Prop}
\begin{proof}
	(1) [sketch].  
	For two points $x,y \in K({\cal L})$, 
	there is a frame ${\cal F}$ such that $K({\cal F})$ contains $x,y$.
    Since $K({\cal F})$ is an isometric subspace of $K({\cal L})$ (Proposition~\ref{prop:basic_K(L)}~(4)),
	the geodesic $[x,y]$ belongs to $K({\cal F})$.
	Hence, a function on 
	$K({\cal L})$ is convex if and only 
	if it is convex on $K({\cal F})$ for every frame ${\cal F}$.
	For any frame ${\cal F}$, 
	the restriction of a submodular function $f:{\cal L} \to \RR$ to ${\cal F}$
	is a usual submodular function on Boolean lattice 
	${\cal F} \simeq 2^{[n]}$.
	Hence
	$\overline f:K({\cal F}) \to \RR$ is viewed as
	the usual Lov\'asz extension by $[0,1]^n \simeq K({\cal F})$, and is convex.
	
	(2). We first show that the restriction $\overline{f}|_{\sigma}$ of $\overline{f}$ to any maximal simplex $\sigma$ is $L$-Lipschitz with 
	$
	L \leq 2 \sqrt{n} \max_{p \in {\cal L}} |f(p)|
	$.
	Suppose that $\sigma$ corresponds to 
	a chain ${\bf 0} = p_0 \prec p_1 \prec \cdots \prec p_n = {\bf 1}$.
	Let $x = \sum_{k}\lambda_k p_k$ and $y = \sum_{k}\mu_k p_k$ be points in $\sigma$. 
	Define $u,v \in \RR^n$ by
	\[
	u_k := \lambda_{k} + \lambda_{k+1} + \cdots + \lambda_{n}, \quad 
	v_k := \mu_{k} + \mu_{k+1} + \cdots + \mu_{n}.
	\]
	By (\ref{eqn:phi_sigma}) and (\ref{eqn:dsigma}), we have 
	$d_{\sigma}(x,y) = \|u-v\|_2.$
	Let $C := \max_{p \in {\cal L}} |f(p)|$. 
	Then we have
	\begin{eqnarray*}
	&& |\overline{f}(x) - \overline{f}(y)| =  \left| 
	\sum_{k=0}^n (\lambda_k - \mu_k) f(p_k) \right| 
	\leq  C\sum_{k=0}^n |\lambda_k - \mu_k|  \\
	&& = C\sum_{k=0}^n | u_{k}- u_{k+1} - (v_{k} - v_{k+1})| \leq
	2 C \sum_{k=1}^n |u_k - v_k|   \leq 2 \sqrt{n} C  \|u-v\|_2, 	
	\end{eqnarray*}
	where we let $u_0 = v_0 := 1$ and $u_{n+1}= v_{n+1} := 0$.
	Thus, $\overline{f}|_{\sigma}$ is $2 \sqrt{n} C$-Lipschitz. 
	
	Next we show that $\overline{f}$ is  $2 \sqrt{n} C$-Lipschitz. 
	For any $x,y\in K({\cal L})$, 
	choose the geodesic $\gamma$ between $x$ and $y$, 
	and $0=t_0 < t_1 < \cdots < t_m = 1$ 
	such that $\gamma([t_i,t_{i+1}])$ belongs  to simplex $\sigma_i$.
	Then we have
	%
	\[
	|\overline f(x) -   \overline f(y)|  \leq  
	\sum_{i=1}^m | \overline{f}(\gamma(t_i)) - \overline f (\gamma(t_{i-1})) |  \leq  2 \sqrt{n} C \sum_{i=1}^{m} d_{\sigma_i}(\gamma(t_i),\gamma(t_{i-1})) =  2 \sqrt{n} C d(x,y).
	\]
	
	(3). Let $f^* :=\min_{p \in {\cal L}} f(p)$, and let $x = \sum_{i} \lambda_i p_i$.
	Suppose to the contrary that all $p_i$'s satisfy $f(p_i) > f^*$.
	Then $f(p_i) \geq f^* + 1$.
	Hence $\overline f(x) = \sum_{i} \lambda_i f(p_i) 
	\geq \sum_{i}\lambda_i (f^*+1) = f^* + 1$.
	However this contradicts $\overline{f}(x) - f^* < 1$.
\end{proof}

\section{Algorithm}\label{sec:algorithm}

%
\subsection{Nc-rank is submodular minimization}\label{subsec:submodular}
Consider MVSP for a linear symbolic matrix $A = \sum_{i=1}^m A_i x_i$.
Let us formulate MVSP as 
an unconstrained submodular function minimization 
over a complemented modular lattice.
Let ${\cal L}$ and ${\cal M}$ 
denote the lattices of all 
vector subspaces of $\mathbb{K}^{n}$, where 
the partial order of ${\cal L}$ is the inclusion order 
and the partial order of ${\cal M}$ is the reverse inclusion order. 
Let $R_{i} = R_{A_{i}}: {\cal L} \times {\cal M} \to \ZZ$ be defined by
\begin{equation*}
R_i (X,Y) := \rank A_i |_{X \times Y} \quad ((X,Y) \in {\cal L} \times {\cal M}),
\end{equation*}
where $A_i |_{X \times Y}: X \times Y \to \mathbb{K}$ is 
the restriction of $A_i$ to $X \times Y$. 
Then the condition $A_i (X,Y) = \{0\}$ in MVSP can be written as 
$R_{i}(X,Y) = 0$.
By using $R_{i}$ as a penalty term, 
consider the following unconstrained problem:
\begin{eqnarray*}
	{\rm MVSP}_{R}: \quad
	{\rm Min.} &&  
	- \dim X - \dim Y
	+ (2n+1) \sum_{i=1}^mR_{i}(X,Y) \\
	{\rm s.t.} && (X,Y) \in {\cal L} \times {\cal M}.
\end{eqnarray*}
Then it is easy to see:
\begin{Lem}
	Any optimal solution of MVSP$_R$ is optimal to MVSP.	
\end{Lem}

\begin{Prop}\label{prop:submodular}
	The objective function of MVSP$_R$ is submodular on ${\cal L} \times {\cal M}$.
\end{Prop}
\begin{proof}
Submodularity of $X \mapsto - \dim X$
and $Y \mapsto - \dim Y$ directly follows from 
$\dim X+ \dim X' = \dim (X \cap X') + \dim (X+X')$.	
Thus it suffices to verify 
that $R = R_i:{\cal L} \times {\cal M} \to \ZZ$ is submodular: 	
\begin{equation*}
R(X,Y) + R(X',Y') \geq R(X \cap X', Y + Y') + R(X + X', Y \cap Y').
\end{equation*}	
Note that 
an equivalent statement appeared in \cite[Lemma 4.2]{IwataMurota95}. 	
	
By Lemma~\ref{lem:frame}, 
there is a base $\{ a_1,a_2,\ldots,a_n\}$ of ${\cal L}$ 
with $X, X', X \cap X', X + X' \subseteq \langle a_1,a_2,\ldots,a_n \rangle$,  
and there is  
a base $\{ b_1,b_2,\ldots,b_n\}$ of ${\cal M}$ 
with $Y, Y', Y \cap Y', Y + Y' \subseteq \langle b_1,b_2,\ldots,b_n \rangle$.
Consider the matrix representation $A = (A(a_i,b_j))$ 
with respect to these bases.
For $I,J \subseteq [n]$, 
let $A[I,J]$ 
be the submatrix of $A$ with row set $I$ and column set~$J$.
Submodularity of $R$ follows from 
the  rank inequality
\begin{equation*}
\rank A[I,J] + \rank A[I',J'] \geq \rank A[I \cap I',J \cup J'] 
+ \rank A[I \cup I',J \cap J'].
\end{equation*}
See~\cite[Proposition 2.1.9]{MurotaBook}.	
\end{proof}
Thus,
MVSP$_R$ has a convex relaxation on CAT(0) space 
$K({\cal L} \times {\cal M}) = K({\cal L}) \times K({\cal M})$ with objective function $g$ that is the Lov\'asz extension
\begin{equation}\label{eqn:g}
g(x,y) := -\overline{\dim}(x) - \overline{\dim} (y) + (2n+1) \sum_{i=1}^{m} \overline{R_i}(x,y).
\end{equation}


\subsection{Splitting proximal point algorithm for nc-rank}
We apply SPPA
to the following perturbed version of the convex relaxation:
\begin{eqnarray*}
	\quad {\rm Min.} &&  
		- \overline{\dim} (x) - \overline{\dim}(y) 
	+ (2n+1) \sum_{i=1}^m \overline{R_{i}}(x,y) + (1/8n) (d(\mathbf{0}, x)^2 + d(\mathbf{0},y)^2)\\
	{\rm s.t.} &&  (x,y) \in K({\cal L}) \times K({\cal M}). 
\end{eqnarray*}
We regard the objective function $\tilde g$ as 
$\sum_{i=1}^{m+2} f_i$, where $f_i$ is defined by
\begin{equation*}
f_i(x,y) := \left\{
\begin{array}{cl}
- \overline{\dim}(x) + (1/8n) d(\mathbf{0},x)^2 & {\rm if}\ 
k = m+1, \\
-  \overline{\dim}(y) + (1/8n) d(\mathbf{0},y)^2 & {\rm if}\ 
k = m+2, \\
(2n+1) \overline{R_{i}}(x, y) & 
{\rm if}\ 1 \leq i \leq m
\end{array}\right.
\end{equation*}
\begin{Thm}\label{thm:prox}
	Let $(z_{\ell})$ be the sequence obtained by SPPA 
	applied to $\tilde g = \sum_{i=1}^{m+2} f_i$ with $\epsilon := 1/2n$.
	For $\ell = \Omega ( n^{12}m^5 \log nm)$, the support of $z_{\ell} =(x_\ell, y_\ell)$ 
contains a minimizer of MVSP.
\end{Thm}
\begin{proof}
	We first show that $f_i$ 
	is $L$-Lipschitz with
	$
	L = O(n^{5/2}).
	$
	By Lemma~\ref{lem:d^2}, Proposition~\ref{prop:basic_K(L)} (4-2), and Proposition~\ref{prop:Lovasz_ext}~(2), 
	the Lipschitz constants of 
		 $\overline{\dim}$ and $d({\bf 0},\cdot)^2$ are $O(n^{3/2})$ and $O(\sqrt{n})$, respectively.
		 Therefore, if $i = m+1$ or $m+2$,
		 then the Lipschitz constant of $f_i$ is $O(n^{3/2})$.
		The Lipschitz constant of other $f_i$ 
		is $O(n^{5/2})$.	
			
The objective function is strongly convex with parameter $1/2n$.
Let $\tilde z$ denote the minimizer of $\tilde g$.
By Theorem~\ref{thm:OhtaPalfia}, we have
\begin{eqnarray*}
&& \tilde g(z_{k(m+2))}) -\tilde g(\tilde z) \leq (m+2) L d(z_{k(m+2)},\tilde z) = 
O\left( \sqrt{\frac{\log k}{k}} n^{6} m^2 \right). 
\end{eqnarray*}
Thus, for $k = \Omega (n^{12} m^4 \log nm)$, it holds
	$
	\tilde g(z_{k(m+2)}) - \tilde g(\tilde z)  < 1/2$.
	
	Let $z^*$ be a minimizer of $g$ of (\ref{eqn:g}).
	Then we have $g(z_{k(m+2)}) - g(z^*) 
	= g(z_{k(m+2)}) - g(\tilde z) + g(\tilde z) - g(z^*) 
	\leq \tilde g(z_{k(m+2)}) - \tilde g(\tilde z) +  (1/8n) d(\mathbf{0},\tilde z)^2 + \tilde g(\tilde z) - \tilde g(z^*) +  (1/8n) d(\mathbf{0},z^*)^2
	\leq  \tilde g(z_{k(m+2)}) - \tilde g(\tilde z) + 1/2 < 1$.
By Proposition~\ref{prop:Lovasz_ext}~(3), the support of $z_{k(m+2)}$ contains a minimizer of MVSP.
\end{proof}
Thus, after a polynomial number of iterations,  
a minimizer $(X^*,Y^*)$ of MVSP exists in the support of $z_{\ell}$.
Our remaining task 
is to show that the resolvent of each summand $f_i$
can be computed in polynomial time.

\subsubsection{Computation of the resolvent for $f_i = - \overline{\dim} + (1/8n) d(\mathbf{0}, \cdot)^2$}
First we consider the resolvent of 
$- \overline{\dim} + (1/8n) d(\mathbf{0}, \cdot)^2$.
This is an optimization problem over the orthoscheme complex of a single lattice.
It suffices to consider the following problem.
\begin{eqnarray*}	
{\rm P1: \quad  Min}. && - \overline{\dim}(x) 
+ \epsilon d(\mathbf{0},x)^2 + \frac{1}{2\lambda} d(x,x^0)^2 \\
{\rm s.t.} && x \in K({\cal L}),
\end{eqnarray*}
where $\epsilon, \lambda > 0$, and $x^0 \in K({\cal L})$.
\begin{Lem}
	Suppose that $x^0$ belongs to a maximal simplex $\sigma$. 
	Then the minimizer $x^*$ of P1 
	exists in $\sigma$.
\end{Lem}
\begin{proof}
	Let $x^0 = \sum_{i=0}^n \lambda_i p_i$ for
	the maximal chain $\{p_i\}$ of $\sigma$.
	Let $x^* = \sum_{i} \mu_i q_i$ be the unique minimizer of P1.
	Consider a frame ${\cal F} = \langle a_1,a_2,\ldots,a_n \rangle$ 
	containing chains $\{p_i\}$ and $\{q_i\}$.
	Notice $K({\cal F}) \simeq [0,1]^n$.
	Let $(x^0_1,x^0_2,\ldots,x^0_n)$ and $(x^*_1,x^*_2,\ldots,x^*_n)$ be 
	the ${\cal F}$-coordinates of $x^0$ and $x^*$, respectively. 
	By~(\ref{eqn:F-coordinate}), 
	 it holds $\overline{\dim}(x) = \sum_i x_i$, 
	since $x = \sum_{k=0}^n \lambda_k a_{i_1} \vee a_{i_2} \vee \cdots \vee a_{i_k} \simeq \sum_{k} \lambda_k 1_{\{i_1,i_2,\ldots,i_k\}}$.
	Hence
	the objective function of P1 is written as
	\[
	-  \sum_{i=1}^n x_i + \epsilon \sum_{i=1}^n x_i^2 + \frac{1}{2\lambda} \sum_{i=1}^n (x_i - x^0_i)^2.
	\] 
	We can assume that $p_i = a_1 \vee a_2 \vee \cdots \vee a_i$ by relabeling.
	Then $x^0_1 \geq x^0_2 \geq \cdots \geq x^0_n$.
	Suppose that $x^0_i > x^0_{i+1}$.
	Then $x^*_i \geq x^*_{i+1}$ must hold. 
	If $x^*_i < x^*_{i+1}$, then interchanging the $i$-coordinate and $(i+1)$-coordinate of 
	$x^*$ gives rise to another point in $K({\cal F})$ having a smaller objective value. This is a contradiction to the optimality of $x^*$.
	Suppose that $x^0_i = x^0_{i+1}$. 
	If  $x^*_i \neq x^*_{i+1}$, 
	then replace both $x_i^*$ and $x_{i+1}^*$ by $(x^*_i+x^*_{i+1})/2$ 
	to decrease the objective value, which is a contradiction.
	Thus $x^*_1 \geq x^*_2 \geq \cdots \geq x^*_n$.
	By~(\ref{eqn:recover}), 
	the original coordinate is written as 
	$x^* = (1- x^{*}_1) {\bf 0} + \sum_{i=1}^n (x^*_i - x^*_{i+1}) (a_{1} \vee a_2 \vee \cdots \vee a_i) = \sum_{i} (x^*_i - x^*_{i+1}) p_i$ 
	(with $x_0^* = 1$ and $x^*_{n+1} = 0$). 
	This means that $x^*$ belongs to $\sigma$.
\end{proof}
As in the proof, to solve P1, consider (implicitly) a frame ${\cal F}$ 
containing the chain $\{p_i\}$ for $x^0 = \sum_i \lambda_i p_i$, 
and the following Euclidean convex optimization problem:
\begin{eqnarray*}
{\rm P1': \quad  Min}. && -  \sum_{i=1}^n x_i + \epsilon \sum_{i=1}^n x_i^2 + \frac{1}{2\lambda} \sum_{i=1}^n (x_i - x^0_i)^2 \\
{\rm s.t.} && 0 \leq x_i \leq 1 \quad (1 \leq i \leq n),
\end{eqnarray*}
where $x$ and $x^0$ are represented in the ${\cal F}$-coordinate.
Then the optimal solution $x^*$ of P1$'$ is obtained coordinate-wise. 
Specifically, $x^*_i$ is $0$, $1$, or $(x_i^0 + \lambda)/(1+ 2\epsilon \lambda)$ for each $i$. According to~(\ref{eqn:recover}), the expression in $K({\cal L})$  is recovered.
\begin{Thm}\label{thm:P1}
	The resolvent of $f_i = - \overline{\dim} + (1/4n) d(\mathbf{0}, \cdot)^2$ 
	is computed in polynomial time.
\end{Thm}
\subsubsection{Computation of the resolvent for $f_i = (2n+1) \overline{R_{i}}$}
Next we consider the computation of 
the resolvent of~$(2n+1) \overline{R_{i}}$. 
It suffices to consider the following problem for $R = R_{A_i}$:
\begin{eqnarray*}	
	{\rm P2: \quad  Min.} && \overline{R}(x,y) +  \frac{1}{2 \lambda} ( d(x,x^0)^2 + d(y,y^0)^2) \\
	{\rm s.t.} && (x,y) \in K({\cal L}) \times K({\cal M}),
\end{eqnarray*}
where $\lambda > 0$, $x^0 \in K({\cal L})$, and $y^0 \in K({\cal M})$.
As in the case of P1, we reduce P2 to a convex optimization 
over $[0,1]^{2n}$
by choosing a special frame $\langle e_1,e_2,\ldots,e_n,f_1,f_2,\ldots,f_n\rangle$ of ${\cal L} \times {\cal M}$.

For $X \in {\cal L}$, let $X^{\bot}$ denote the subspace in ${\cal M}$ defined by
\begin{equation*}
X^{\bot} := \{ y \in \KK^n \mid A_i(x,y) = 0 \  (x \in X) \}.
\end{equation*}
Namely $X^{\bot}$ is the orthogonal subspace of $X$ with respect to the bilinear form $A_i$.
For $Y \in {\cal M}$, let $Y^{\bot} \in {\cal L}$ be defined analogously.
Let $U_0 \in {\cal L}$ and $V_0 \in {\cal M}$ denote the left  and right kernels of $A_i$, respectively:
\begin{eqnarray*}
	U_0 &:= &\{ x \in \KK^n \mid A_i(x,y) = 0 \  (y \in \KK^n) \}. \\ 
	V_0 &:= & \{ y \in \KK^n \mid A_i(x,y) = 0 \  (x \in \KK^n) \}. 
\end{eqnarray*}

Let $k := \rank A_i$.
An {\em orthogonal frame} ${\cal F} = \langle e_1,e_2,\ldots,e_n,f_1,f_2,\ldots,f_n\rangle$
is a frame of ${\cal L} \times {\cal M}$ satisfying the following conditions:
\begin{itemize}
\item $\langle e_1,e_2,\ldots,e_n \rangle$ is a frame of ${\cal L}$.
\item $\langle f_1,f_2,\ldots,f_n\rangle$ is a frame of ${\cal M}$.
\item $e_{k+1} \vee e_{k+2} \vee \cdots \vee e_n = U_0$.
\item $f_1 \vee f_2 \vee \cdots \vee f_k = V_0$ ($\Leftrightarrow$ $f_1 \cap f_2 \cap \cdots \cap f_k = V_0$ ).
\item $f_i = {e_i}^{\bot}$ for $i=1,2,\ldots,k$.
\end{itemize}
Figure~\ref{fig:frame} is an intuitive illustration of an orthogonal frame.
	\begin{figure}[t]
	\begin{center}
		\includegraphics[scale=0.45]{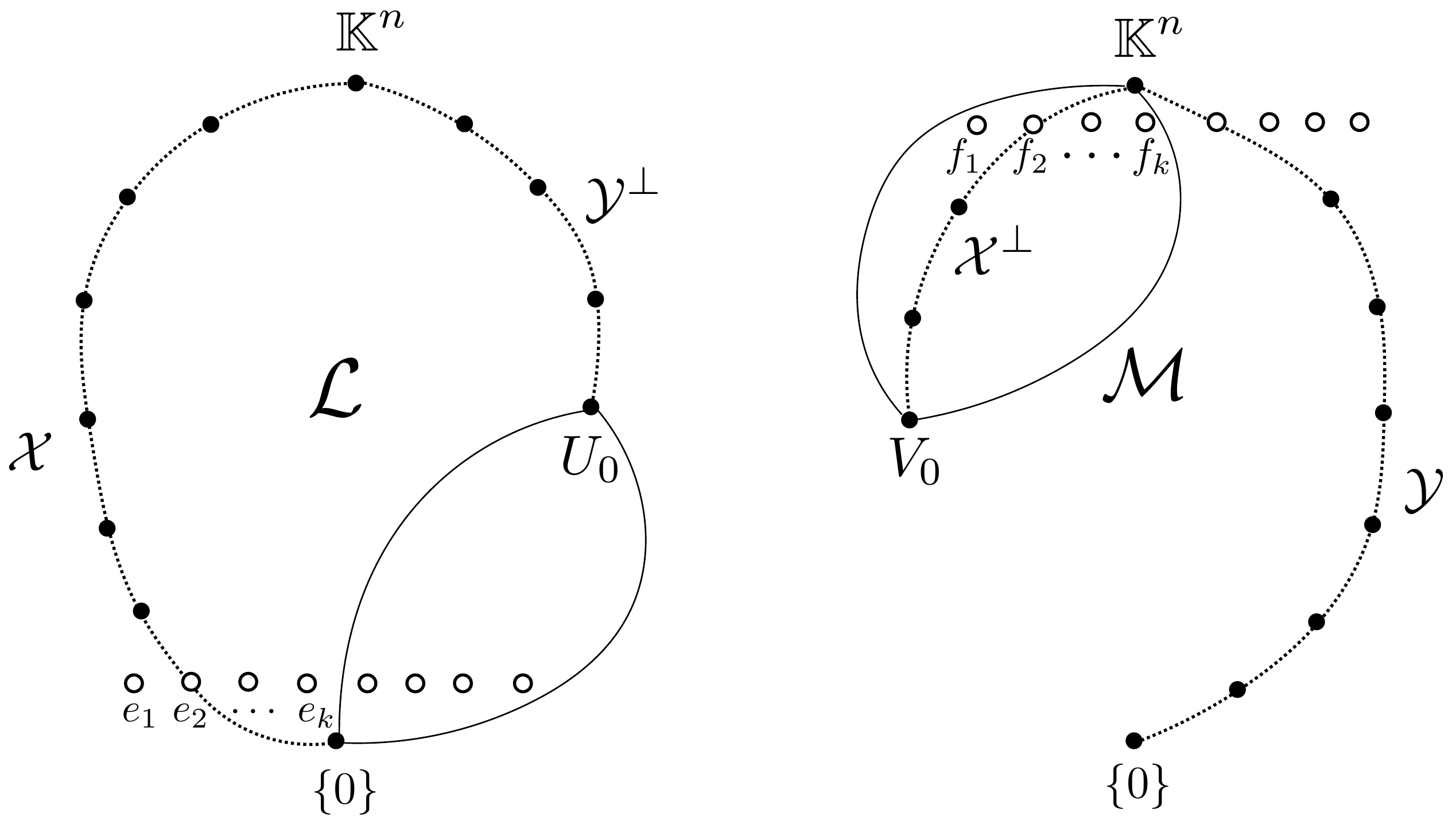}
		\caption{An orthogonal frame}
		\label{fig:frame}
	\end{center}
\end{figure}\noindent

\begin{Prop}\label{prop:Lovasz}
	Let  ${\cal F} = \langle e_1,e_2,\ldots,e_n,f_1,f_2,\ldots,f_n\rangle$ be an orthogonal frame.
	The restriction of the Lov\'asz extension $\overline{R}$ to 
	$K({\cal F}) \simeq [0,1]^n \times [0,1]^n$  can be written as
	\begin{equation}
	\overline{R} (x,y) = \sum_{i=1}^k \max\{0, x_i - y_i\},
	\end{equation}
	where $(x_1,x_2,\ldots,x_n)$ is the $\langle e_1,e_2,\ldots,e_n \rangle$-coordinate of $x$
	and $(y_1,y_2,\ldots,y_n)$ is the $\langle f_1,f_2,\ldots,f_n \rangle$-coordinate of $y$.
\end{Prop}
\begin{Prop}\label{prop:orthogonal_frame}
	Let ${\cal X}$ and ${\cal Y}$ be maximal chains of ${\cal L}$ and ${\cal M}$, respectively. 
	Then there exists an orthogonal frame ${\cal F} = \langle e_1,e_2,\ldots,e_n,f_1,f_2,\ldots,f_n\rangle$ satisfying
	\begin{equation}\label{eqn:XuYbot}
	{\cal X} \cup {\cal Y}^{\bot} \subseteq 
	\langle e_1,e_2,\ldots,e_n \rangle,\   {\cal X}^{\bot} \cup {\cal Y} \subseteq 
	\langle f_1,f_2,\ldots,f_n \rangle.  
	\end{equation}
	Such a frame can be found in polynomial time.	
\end{Prop}

 \begin{Prop}\label{prop:R(x,y)}
	 Let ${\cal X}$ and ${\cal Y}$ 
	 be maximal chains corresponding to 
	 maximal simplices containing $x^0$ and $y^0$, respectively.
    For an orthogonal frame ${\cal F}$ satisfying $(\ref{eqn:XuYbot})$, 
    the minimizer $(x^*,y^*)$ of P2 exists in $K({\cal F})$.
\end{Prop}

The above three propositions are proved in Section~\ref{subsub:proof}.
Assuming these, we proceed the computation of the resolvent.
For an orthogonal frame satisfying~(\ref{eqn:XuYbot}), the problem P2 is equivalent to
\begin{eqnarray*}	
	{\rm P2': \quad  Min.} && \sum_{i=1}^k \max\{0,x_i - y_i \} + \frac{1}{2 \lambda} \left\{\sum_{i=1}^m (x_i -x^0_i)^2 + \sum_{i=1}^n (y_i - y^0_i)^2 \right\} \\
	{\rm s.t.} && 0 \leq x_i \leq 1, 0 \leq y_i \leq 1  \quad (0 \leq i \leq n).
\end{eqnarray*}
Again this problem is easily solved coordinate-wise.
Obviously $x^*_i = x^0_i$ and $y^*_i = y^0_i$ for $i > k$.
For $i \leq k$,  
$(x^*_i,y^*_i)$ 
is the minimizer of the $2$-dimensional problem.
Obviously this can be solved in constant time.
%
%
%
\begin{Thm}\label{thm:P2}
	The resolvent of $f_i = (2n+1)\overline{R_i}$ is computed in polynomial time.
\end{Thm}

\begin{Rem}[Bit complexity]\label{rem:bit}
	In the above SPPA, 
	the required bit-length for coefficients of $z \in K({\cal L} \times {\cal M})$ 
	is bounded polynomially in $n,m$. Indeed,
	the transformation between the original coordinate and an ${\cal F}$-coordinate
	corresponds to multiplying a triangular matrix consisting of $0,\pm 1$ entries; see (\ref{eqn:recover}). 
	In each iteration $k$,
	the optimal solution of quadratic problem P1$'$ or P2$'$ 
	is obtained by 
	adding (fixed) rational functions in $n,m,k$ 
	to (current points) $x_i^0, y_i^0$ 
	and multiplying a (fixed) $2 \times 2$ rational matrix in $n,m,k$.   
	Consequently, the bit increase is polynomially bounded.

	On the other hand, in the case of $\KK = \QQ$, we could not exclude the possibility of an exponential increase of the bit-length for the basis of a vector subspace appearing in the algorithm. 
\end{Rem}

\subsubsection{Proofs of Propositions~\ref{prop:Lovasz}, \ref{prop:orthogonal_frame}, and \ref{prop:R(x,y)}}\label{subsub:proof}
We start with basic properties of $(\cdot)^{\bot}$, which follow from elementary linear algebra.
\begin{Lem}\label{lem:bot}
	\begin{itemize}
		\item[{\rm (1)}] If $X \subseteq X'$, then $X^{\bot} \supseteq {X'}^{\bot}$ and $\dim {X}^{\bot} - \dim {X'}^{\bot} \leq \dim X'-\dim X$.
		\item[{\rm (2)}] $(X + X')^{\bot} = X^{\bot} \cap {X'}^{\bot}$. 
		\item[{\rm (3)}] $X^{\bot \bot} \supseteq X$.
		\item[{\rm (4)}] $X \mapsto X^{\bot}$ 
		induces an isomorphism between $[U_0, \KK^n]$ and 
		$[\KK^n,V_0]$ with inverse $Y \mapsto Y^{\bot}$. In particular, 
		$X^{\bot \bot \bot} = X^{\bot}$.
	\end{itemize}
\end{Lem}

An alternative expression of $R$ by using $(\cdot)^{\bot}$ is given.
\begin{Lem}\label{lem:formula_R}
	$R(X,Y) = \dim Y - \dim Y \cap X^{\bot} = \dim X - \dim X \cap Y^{\bot}$.
\end{Lem}
\begin{proof}
	Consider bases $\{a_1,a_2,\ldots,a_{\ell}\}$ of $X$ and $\{ b_1,b_2,\ldots,b_{\ell'}\}$ of $Y$.
	We can assume that
	$\{a_{k'+1},a_{k'+2},\ldots,a_{\ell'}\}$ is a base of 
	$Y \cap X^{\bot}$.
	Consider the matrix representation $(A_i (a_{i'},b_{j'}))$ of $A_i|_{X \times Y}$ 
	with respect to these bases. 
	Its submatrix of $k'+1,k'+2,\ldots,\ell'$-th columns 
	is a zero matrix.
	On the other hand, 
	the submatrix of $1,2,\ldots, k'$-th columns
	must have the column-full rank $k'$. 
	Thus, the rank $R(X,Y)$ of $A_i|_{X \times Y}$ is 
	$k' = \ell' - (\ell'-k') = r(Y) - r(Y \cap X^{\bot})$.
	The second expression is obtained similarly.
\end{proof}

\paragraph{Proof of Proposition~\ref{prop:Lovasz}.}
An orthogonal frame 
$\langle e_1,e_2,\ldots,e_n,f_1,f_2,\ldots,f_n \rangle$ 
is naturally identified with Boolean lattice 
$2^{[2n]} \simeq 2^{[n]} \times 2^{[n]}$.
Notice that ${e_i}^{\bot} = f_i$ if $i \leq k$ 
and ${e_i}^{\bot} = \KK^n$ if $i > k$.
The latter fact follows from 
$e_i \subseteq U_0 \Rightarrow {e_i}^{\bot} \supseteq U_0^{\bot} = \KK^n$.
By Lemma~\ref{lem:bot}~(2), 
we have
$X^{\bot} = X \cap \{1,2,\ldots,k\}$ for $X \in 2^{[n]}$.
By Lemma~\ref{lem:formula_R} 
and $\dim Y = n - |Y|$ for 
$Y \in 2^{[n]} \simeq \langle f_1,f_2,\ldots,f_n \rangle$ 
(with inclusion order reversed),
 we have
\[
	R(X,Y)  =  |Y \cup (X \cap [k]) | - |Y| =  |(X \setminus Y) \cap [k]|.
\]
Identify $2^{[n]} \times 2^{[n]}$ with $\{0,1\}^n \times \{0,1\}^n$ 
by $(X,Y) \mapsto (1_{X}, 1_{Y})$.
Then $R$ is also written as
\begin{equation*}
R(x,y) = \sum_{i=1}^k \max \{0, x_i - y_i\} \quad ((x,y) \in \{0,1\}^n \times \{0,1\}^n).
\end{equation*}
Observe that the Lov\'asz extension of 
$(x_i,y_i) \mapsto \max \{0, x_i - y_i\}$ is 
obtained simply by extending the domain to $[0,1]^2$.
Hence, we obtain the desired expression.

\paragraph{Proof of Proposition~\ref{prop:orthogonal_frame}.}
By Lemma~\ref{lem:frame}, 
we can find (in polynomial time) a frame $\langle e_1,e_2,\ldots,e_n \rangle$ 
containing two chains ${\cal X}$ and ${\cal Y}^{\bot}$.
Suppose that ${\cal X} = \{X_i\}_{i=0}^n$ and ${\cal Y} = \{Y_i\}_{i=0}^n$.
We can assume that $e_{k+1} \vee e_{k+2} \vee \cdots \vee e_{n} = {Y_0}^{\bot} = U_0$.
Let $f_i := {e_i}^{\bot}$ for $i=1,2,\ldots,k$.
Then $f_1 \vee f_2 \vee \cdots \vee f_n = V_0$ holds,
since, by Lemma~\ref{lem:bot}~(2), we have 
$V_0 =  (e_1 \vee e_2 \vee \cdots \vee e_n)^{\bot} = {e_1}^{\bot} \vee {e_2}^{\bot} \vee \cdots \vee {e_n}^{\bot} = f_1 \vee f_2 \vee \cdots f_k \vee \KK^n \vee \cdots \vee \KK^n = f_1 \vee f_2 \vee \cdots \vee f_k$.

Consider the chain ${\cal Y}^{\bot \bot}$ in ${\cal M}$.
Then ${\cal Y}^{\bot \bot} \subseteq \langle f_1,f_2,\ldots,f_k \rangle$ since
each $Y_i^{\bot}$ is the join of a subset of $e_1,e_2,\ldots,e_n$.
Taking $(\cdot)^\bot$ as above, 
$Y_i^{\bot \bot}$ is represented as the join of a subset of $f_1,f_2,\ldots,f_k$.
Consider a consecutive pair $Y_{i-1}, Y_{i}$ in ${\cal Y}$.
Consider ${Y_{i-1}}^{\bot \bot}$ and ${Y_{i}}^{\bot \bot}$.
Then, by Lemma~\ref{lem:bot}~(3), 
${Y_{i-1}}^{\bot \bot} \preceq Y_{i-1}$ and ${Y_{i}}^{\bot \bot} \preceq Y_{i}$.
Suppose that ${Y_{i-1}}^{\bot \bot} \neq {Y_{i}}^{\bot \bot}$.
Then ${Y_{i-1}}^{\bot \bot} \prec: {Y_{i}}^{\bot \bot}$ (by (\ref{eqn:basic}) and Lemma~\ref{lem:bot}~(1)).
Thus, for some $f_j$ $(1 \leq j \leq k)$, 
it holds ${Y_{i}}^{\bot \bot} = f_j \vee {Y_{i-1}}^{\bot \bot}$.
Here $f_j \not \preceq Y_{i-1}$ must hold.
Otherwise ${Y_{i-1}}^{\bot \bot} \succeq {f_j}^{\bot \bot}  = {e_j}^{\bot \bot \bot} = f_j$, 
which contradicts ${Y_{i-1}}^{\bot \bot} \prec: {Y_{i}}^{\bot \bot} = f_j \vee {Y_{i-1}}^{\bot \bot}$.
Also, $f_j \preceq Y_i^{\bot \bot} \preceq Y_i$.
Thus $Y_{i} = Y_{i-1} \vee f_j$.
Therefore, for each $i$ with ${Y_{i-1}}^{\bot \bot} = {Y_{i}}^{\bot \bot}$, 
we can choose an atom $f$ with $Y_i = f \vee Y_{i-1}$ to add to $f_1,f_2,\ldots,f_k$, 
and obtain a required frame 
$\langle f_1,f_2,\ldots f_n \rangle$ (containing ${\cal X}^{\bot}$ and ${\cal Y}$).

\paragraph{Proof of Proposition~\ref{prop:R(x,y)}.}
Consider retractions $\varphi 
:= \varphi_{e_n,e_{n-1},\ldots,e_{k+1},e_1,e_2,\ldots,e_k}: {\cal L} \to \langle e_1,e_2,\ldots,e_n \rangle$ 
and $\phi :=  \varphi_{f_1,f_2,\ldots,f_n}: {\cal M} \to \langle f_1,f_2,\ldots,f_n \rangle$; 
see Lemma~\ref{lem:frame}~(2) for definition.
Define a retraction $(\bar \varphi, \bar \phi): K({\cal L}) \times K({\cal M}) \to K(\langle e_1,\ldots,e_n,f_1,\ldots,f_n \rangle)$ by
\begin{equation*}
(\bar \varphi, \bar \phi)(x,y) := (\bar\varphi(x),\bar\phi(y)) \quad ((x,y) \in  K({\cal L}) \times K({\cal M}))
\end{equation*} 
Our goal is to show that $(\bar \varphi, \bar \phi)$
does not increase the objective value of P2.

First we show
\begin{equation}\label{eqn:phi_bot}
(\phi(Y))^{\bot}  = \varphi(Y^{\bot})  \quad (Y \in {\cal M}).
\end{equation}
Indeed, letting $F_i := f_1 \vee f_2 \vee \cdots \vee f_i$ and $E_i := e_1 \vee e_2 \vee \cdots \vee e_i$, we have
\begin{eqnarray}
(\phi(Y))^{\bot} & = & \left( \bigvee \{f_i \mid i \in [n]: Y \wedge F_i :\succ Y \wedge F_{i-1}  \}     \right)^{\bot} \nonumber \\
& = & \left( V_0  \wedge \bigvee \{f_i \mid i \in [n]: Y \wedge F_i :\succ Y \wedge F_{i-1}  \}   \nonumber  \right)^{\bot}  \nonumber \\
& =&  \left( \bigvee \{f_i \mid i \in [k]: Y \wedge F_i :\succ Y \wedge F_{i-1}  \}   \nonumber  \right)^{\bot} \nonumber \\
& =&  \left( \bigvee \{f_i \mid i \in [k]: (Y \wedge V_0) \wedge F_i :\succ (Y \wedge V_0)  \wedge F_{i-1}  \}   \nonumber  \right)^{\bot} \nonumber \\
& =&  \bigvee \{U_0 \vee e_i \mid i \in [k]: Y^{\bot} \wedge (U_0  \vee E_i) :\succ Y^{\bot} \wedge (U_0  \vee E_{i-1})  \}   = \varphi(Y^{\bot}). \nonumber
\end{eqnarray}
The second equality follows from 
$(V_0 + Z)^{\bot} = V_0^{\bot} \cap Z^{\bot} = \KK^n \cap Z^{\bot} = Z^\bot$.
The third from the modularity:
Let $A := \bigvee \{f_i \mid i \in [k]: Y \wedge F_i :\succ Y \wedge F_{i-1}  \}$ and 
$B := \bigvee \{f_i \mid i \in [n] \setminus [k]: Y \wedge F_i :\succ Y \wedge F_{i-1}  \}$. 
Then $V_0 \wedge B = \KK^n$ and $Y = A \vee B$. 
Thus we have $A = (V_0 \wedge B) \vee A = V_0 \wedge Y$.
The forth follows from $f_i \preceq V_0$ for $i \in [k]$.
The fifth follows from Lemma~\ref{lem:bot}~(4). Note 
that by $Y^{\bot} \succeq U_0$ 
each atom $e_i$ with $i \geq k+1$ is taken in the join of the definition (\ref{eqn:retraction}) of $\varphi = \varphi_{e_n,e_{n-1},\ldots,e_{k+1},e_1,e_2,\ldots,e_k}$.
 
Next we show
\begin{equation}\label{eqn:R_phi}
R(\varphi(X), \phi(Y)) \leq R(X, Y) \quad (X \in {\cal L}, Y \in {\cal M}).
\end{equation}
Indeed, for $r = \dim$, we have
$
R(\varphi(X), \phi(Y)) 
= r(\varphi(X)) - r(\varphi(X) \wedge \phi(Y)^{\bot})
= r(X) - r(\varphi(X) \wedge \varphi(Y^{\bot}))
\leq  r(X) - r(\varphi(X \wedge Y^{\bot}))
=  r(X) - r(X \wedge Y^{\bot}) = R(X,Y).
$
In the second equality, 
we use (\ref{eqn:phi_bot}) and rank-preserving property of $\varphi$. 
The inequality follows from order-preserving property 
$\varphi(X) \wedge \varphi(Y^{\bot}) \succeq \varphi(X \wedge Y^{\bot})$.

By (\ref{eqn:R_phi}), 
we have $R(\bar \varphi(x), \bar \phi(y)) \leq R(x,y)$; 
recall the isometry between $K({\cal L} \times {\cal M})$
and $K({\cal L}) \times K({\cal M})$ (Section~\ref{subsub:K(L)}).
 Since $\bar \varphi$ and $\bar \phi$ are nonexpansive retractions (Proposition~\ref{prop:basic_K(L)})),
we have $d(x^0,x) \geq d(\bar \varphi(x^0), \bar \varphi(x)) = d(x^0, \bar \varphi(x))$ and $d(y^0,y) \geq d(\bar \phi(y^0), \bar \phi(y)) = d(y^0, \bar \phi(y))$.
Thus, $(\bar \varphi, \bar \phi)$ has a desired property 
to prove the statement.

\section{A $p$-adic approach to nc-rank over $\QQ$}\label{sec:p-adic}

In this section, we consider 
nc-rank computation of $A = \sum_{i=1}^m A_i x_i$, where
each $A_i$ is a matrix over $\QQ$. 
Specifically, we assume that each $A_i$ is an integer matrix.
As remarked in Remark~\ref{rem:bit}, 
the algorithm in the previous section has no polynomial guarantee for the length of bits representing bases of vector subspaces.
Instead of controlling bit sizes,
we consider to reduce 
nc-rank computation over $\QQ$ to
that over $GF(p)$ (for small $p$).

For simplicity, we deal with 
nc-singularity testing of $A$. 
Here $A$ is called {\em nc-singular} if $\ncrank A < n$, 
and called {\em nc-regular} if $\ncrank A = n$.
We utilize a relationship between nc-rank and the ordinary rank (on arbitrary field $\KK$).
For a positive integer $d$, the {\em $d$-blow up} $A^{\{d\}}$ of $A$ is a linear symbolic matrix defined by
\begin{equation*}
A^{\{d\}} := \sum_{i=1}^m A_i \otimes X_i,
\end{equation*}
where $\otimes$ denotes the Kronecker product and
$X_i = (x_{i,jk})$ is a $d \times d$ matrix with variable entries 
$x_{i,jk}$ $(i \in [m],j,k \in [d])$.
\begin{Lem}[\cite{HrubesWigderson2015,Kaliuzhnyi-VerbovetskyiVinnikov2012}]\label{lem:blowup}
	A matrix $A$ of form (\ref{eqn:A})
	is nc-regular if and only if there is a positive integer $d$ such that $A^{\{d\}}$ is regular. 
\end{Lem}
There is an upper bound of such a $d$. 
Derksen and Makam\cite{DerksenMakam17} proved 
a polynomial (linear) bound $d \leq n-1$ 
by utilizing the {\em regularity lemma} $d| \rank A^{\{d\}}$ in \cite{IQS15a}.
Such bounds play an essential role in the validity of the algorithms of~\cite{GGOW15,IQS15a,IQS15b}.
Interestingly, our reduction presented below does not use any bound of $d$.

Fix an arbitrary prime number $p > 1$. 
Let $v_p: \QQ  \to \ZZ \cup \{\infty\}$ denote the $p$-adic valuation:
\begin{equation*}
v_p(u) := k\  {\rm if}\ u = p^{k} a/b, 
\end{equation*}
where $a, b$ are nonzero integers prime to $p$, and we let $v_p(0) := \infty$.
Every rational $u \in \QQ$ is uniquely represented as the $p$-adic expansion 
\begin{equation}\label{eqn:p-adic}
u = \sum_{i=k}^{\infty} a_i p^i,
\end{equation}
where $k= v_p(u)$ and $a_i \in \{0,1,2,\ldots,p-1\}$.
The leading (nonzero) coefficient $a_k$ is given as the solution of $a = b x \mod p$. Then $u - a_k p^{k}$ is divided 
by $p^{k+1}$. Repeating the same procedure for $u - a_k p^{k}$, we obtain the subsequent coefficients in (\ref{eqn:p-adic}). 

The $2$-adic expansion of a nonnegative integer $z$ 
is the same as the binary expression of $z$, 
where $v_2(z)$ is equal to 
the number of consecutive zeros from the first bit.  
This interpretation holds for an arbitrary prime $p$.
In particular, 
the $p$-adic valuation of a nonzero integer is bounded by 
the bit-length in base $p$:
\begin{equation}\label{eqn:vp_vs_bit}
v_p (z) \leq \log_p |z|  \quad (z \in \ZZ \setminus \{0\}).
\end{equation} 


The $p$-adic valuation $v_p$ on $\QQ$ is extended to $\QQ(x_1,x_2,\ldots,x_m)$ as follows.
For a polynomial $f \in \QQ[x_1,x_2,\ldots,x_m]$, define $v_p(f)$ by  
\begin{equation}\label{eqn:Gauss}
v_p (f) := \min \{ v(a) \mid \mbox{$a$ 
	is the coefficient of a term of $f$}\}.
\end{equation}
Accordingly, the valuation of 
a rational function $f/g$ is defined as $v_p (f)-v_p(g)$. 
This is called  the {\em Gauss extension} of~$v_p$.

Our algorithm for testing nc-singularity is 
based on the following problem ({\em maximum vanishing submodule problem; MVMP}):  
\begin{eqnarray*}
	\mbox{MVMP:} \quad {\rm Max}. &&  -v_p \det P - v_p \det Q \\
	{\rm s.t.}  && v_p (PA Q)_{ij} \geq 0 \quad (i,j \in [n]), \\
	&& P,Q \in GL_n (\QQ).
\end{eqnarray*}
This problem is definable for 
an arbitrary field with a discrete valuation, and the following arguments are applicable for such a field, while 
\cite{HH_degdet} introduced MVMP
for the rational function field with one valuable.

MVMP is also a discrete convex optimization 
on a CAT(0) space. Indeed, 
its domain can be viewed as 
the vertex set (the set of {\em lattices}, certain submodules of $\QQ^n$) of 
the Euclidean building for $GL_n(\QQ)$, and the objective function is an {\em L-convex function}; see~\cite{HH_degdet,HI_degdet}.
A Euclidean building is a representative space admitting a CAT(0)-metric.

The optimal value of MVMP is denoted by 
$v_p \Det' A\in \ZZ \cup \{\infty\}$, where we let $v_p \Det' A := \infty$
if MVMP is unbounded. 
The motivation behind this notation $v_p \Det' A$ 
is explained in Remark~\ref{rem:Det}.

For a feasible solution $(P,Q)$ of MVMP,  
consider the $p$-adic expansion of 
$PA_i Q = \sum_{k=0}^{\infty} (PA_iQ)^{(k)} p^k$ for each $i$.
The leading matrix 
$(PA_i Q)^{(0)}$ consists of values $0,1,\ldots,p-1$ and is considered in $GF(p)$.
Then we can consider the linear symbolic matrix 
\[
(PAQ)^{(0)} := \sum_{i=1}^m (PA_i Q)^{(0)} x_i 
\]
over $GF(p)$.
\begin{Lem}\label{lem:Det<=det}
	For a feasible solution $(P,Q)$ of MVMP, the following hold:
	\begin{itemize}
		\item[{\rm (1)}] $- v_p \det P -  v_p \det Q \leq v_p \det A$. In particular, $v_p \Det' A \leq v_p \det A$.
		\item [{\rm (2)}]	If $(PAQ)^{(0)}$ is regular, 
		then $v_p \det A = - v_p \det P - v_p \det Q = v_p \Det' A$.
	\end{itemize}
\end{Lem}
\begin{proof}
They follow from $
0 \leq v_p \det P A Q =  v_p \det P +  v_p \det Q + v_p \det A$.
The inequality holds in equality precisely 
when the leading matrix $(PAQ)^{(0)}$ 
is regular.
\end{proof}
The following algorithm for MVMP is due to \cite{HH_degdet}, 
which originated from Murota's {\em combinatorial relaxation algorithm}~\cite{MurotaBook} 
and can be viewed as an descent algorithm on the Euclidean building. For an integer vector $z \in \ZZ$, 
let $(p^{z})$ denote the diagonal matrix with diagonals 
$p^{z_1}, p^{z_2},\ldots,p^{z_n}$ in order.
\begin{description}
	\item[Algorithm: Val-Det]
	\item[0:] Let $(P,Q) := (I,I)$.
	\item[1:]  Solve FR (or MVSP) for $(PAQ)^{(0)}$, and 
	obtain optimal matrices $S,T \in GL_n(GF(p))$ such that 
	$S(PAQ)^{(0)}T$ has an $r \times s$ zero submatrix in its upper-left corner.
	\item[2:] If  $(PAQ)^{(0)}$ is nc-singular, i.e.,
	$n < r+s$, then 
	let $(P,Q) \leftarrow ((p^{- 1_{[r]}})SP, QT (p^{1_{[n] \setminus [s]}}))$
	and go to step 1. Otherwise stop.
\end{description}

The initial $(P,Q)$ in step 0 is feasible with objection value $0$, as each $A_i$ is an integer matrix.
In step 2, $S, T$ are regarded as matrices in $GL_n(\QQ)$ with entries 
in $\{0,1,\ldots,p-1\}$.
Observe that each entry in the $r \times s$ upper-left submatrix of $SPAQT$
is divided by $p$.
Thus, the update in step 2 keeps the feasibility of $(P,Q)$.
Further, it strictly increases the objective value:
$- v_p \det (p^{-1_{[r]}}) S P - v_p \det QT(p^{1_{[n] \setminus [s]}}) 
= (r + s - n) - v_p \det P - v_p \det Q$.
Note that $\det S$ and $\det T$ cannot be divided by $p$,
since $S$ and $T$ are invertible in modulo $p$.  
Therefore, nc-regularity of $(PAQ)^{(0)}$ 
is a necessary condition for optimality of~$(P,Q)$. 
In fact, it is sufficient.
\begin{Prop}[{\cite{HI_degdet}}]\label{prop:optimality} 
	A feasible solution $(P,Q)$ is optimal if and only if $(PAQ)^{(0)}$ is nc-regular. In this case, 
		it holds $v_p \Det' A 
		= (1/d) v_p \det A^{\{d\}}$ for some $d > 0$.
\end{Prop}
\begin{proof}
	As in \cite[Lemma 4.2 (1)]{HI_degdet}, one can show
	$v_p \Det' A = (1/d) v_p \Det' A^{\{d\}}$ for all $d$.
	By Lemma~\ref{lem:Det<=det}, $v_p \Det' A \leq (1/d) v_p \det A^{\{d\}}$ holds for all $d$. 
	
	Suppose that $(PAQ)^{(0)}$ is nc-regular. 
	It suffices to show $v_p \Det' A \geq (1/d) v_p \det A^{\{d\}}$ for some $d$.
	By Lemma~\ref{lem:blowup}, 
	for some $d > 0$, 
	$((PAQ)^{(0)})^{\{d\}} = ((PAQ)^{\{d\}})^{(0)} 
	= ((P \otimes I) A^{\{d\}} (Q \otimes I))^{(0)}$ is regular. Observe that 
	$(P \otimes I, Q \otimes I)$ is feasible 
	to MVMP for $A^{\{d\}}$. By Lemma~\ref{lem:Det<=det}~(2),   
	we have $v_p \det A^{\{d\}} =
	 - v_p \det P \otimes I - v_p \det Q \otimes I 
	 = - d (v_p \det P + v_p \det Q) \leq d v_p \Det' A$. 
\end{proof}

From the proof and Lemma~\ref{lem:blowup}, we have:
\begin{Cor}\label{cor:regularity}
	$A$ is nc-regular if and only if $v_p \Det' A < \infty$.
\end{Cor}

Therefore, {\bf Val-Det} does not terminate if $A$ is nc-singular. 
A stopping criterion guaranteeing nc-singularity of $A$ is obtained as follows:
\begin{Prop}\label{prop:bound}
	Suppose that each $A_i$ consists of integer entries whose absolute values are at most $D$.
	If $A$ is nc-regular, then $v_p \Det' A =O(n \log_p nD)$.
	Thus, $\Omega(n \log_p nD)$ iterations of 
	{\bf Val-Det} certify nc-singularity of $A$. 
\end{Prop}
\begin{proof}
Suppose that $A$ is nc-regular. By Proposition~\ref{prop:optimality}, $v_p \Det' A = (1/d)v_p \det A^{\{d\}}$ for some $d$.
We estimate $v_p \det A^{\{d\}}$. The following argument is a sharpening of the proof of \cite[Lemma 4.9]{HI_degdet}. 
Rewrite $A^{\{d\}}$ as
\[
A^{\{d\}} = \sum_{i \in[m],j,k\in [d]} A_{i,jk} x_{i,jk},
\]
where $A_{i,jk}$ is 
an $nd \times nd$ block matrix with block size $n$ such that 
the $(j,k)$-th block equals to $A_i$ 
and other blocks are zero.
By multilinearity of determinant, we have
\[
\det A^{\{d\}} = \sum_{\alpha_1,\alpha_2,\ldots,\alpha_{nd}}  \pm \det A[\alpha_1,\alpha_2,\ldots,\alpha_{nd}] x_{\alpha_1}x_{\alpha_2}\cdots  x_{\alpha_{nd}},
\]
where 
$\alpha_\gamma$ $(\gamma \in [nd])$ ranges over 
$\{(i,jk)\}_{i \in [m], j\in [d]}$ if 
$\gamma$ belongs to the $k$-th block (i.e., $k= \lceil \gamma/d \rceil)$ and
$A[\alpha_1,\alpha_2,\ldots,\alpha_{nd}]$ is the $nd \times nd$ matrix with the $\gamma$-th column chosen from   
$A_{k,ij}$ with $\alpha_\gamma = (k,ij)$.
A monomial in this expression is written as 
$a_z \prod x_{k,ij}^{z_{k,ij}}$ for a nonnegative vector $z = (z_{i,jk}) \in \ZZ^{md^2}$ with $\sum_{i,j} z_{i,jk} = n$ $(k \in [d])$. The coefficient $a_z$ is given by  
\[
a_z = \sum_{\alpha_1,\alpha_2,\ldots,\alpha_{nd}} \pm \det A[\alpha_1,\alpha_2,\ldots,\alpha_{nd}],
\]
where $\alpha_1,\alpha_2,\ldots,\alpha_{nd}$ are taken so that 
$(i,jk)$ appears $z_{i,jk}$ times. 
The total number of such indices is
\[
\prod_{k=1}^d \frac{n!}{\prod_{i,j} z_{i,jk} !} 
\leq n^{nd}. 
\]
From Hadamard's inequality and the fact 
that each column of $A[\alpha_1,\alpha_2,\ldots,\alpha_{nd}]$ has at most $n$ nonzero entries with absolute values at most $D$, 
we have
\begin{equation}
|a_z| \leq  n^{nd}  (n^{1/2}D)^{nd}
\leq n^{3nd/2} D^{nd}.
\end{equation}
Therefore, the bit length of $a_z$ in base $p$ is bounded by $O( nd \log_p n D)$.
By (\ref{eqn:vp_vs_bit}), we have $v_p \det A^{\{d\}} = O(nd \log n D)$. 
Thus, $v_p \Det' A = O(n \log_p n D)$.
\end{proof}

For $p=2$, the algorithm {\bf Val-Det} is executed as follows.
Instead of updating $(P,Q)$, update $A$ as 
$A \leftarrow (p^{-1_{[r]}})SAT(p^{1_{[n] \setminus [s]}})$.
Then, $A^{(0)}$ is computed as $(A_i)^{(0)} =  A_i \mod 2$.
In step 2, $S,T$ are $0,1$ matrices 
such that all entries of the $r \times s$ corner of each $SA_iT$ 
are divided by $2$. 
Hence, the next $A_i$ is again an integer matrix.
The bit-length bound of each entry in $A_i$ increases 
by $O(\log_2 n)$ (starting from the initial bound $O(\log_2 D)$).
Therefore, until detecting nc-singularity of $A$, 
the required bit-length is $O(n \log_2 n \log_2 n D)$.

\begin{Rem}[Valuations on the free skew field]\label{rem:Det}
As shown by Cohn~\cite[Corollary 4.6]{Cohn_valuation}, 
any valuation $v$ on a field $\KK$ is extended to 
the free skew field $\KK(\langle x_1,\ldots,x_m \rangle)$.
Then we can consider the valuation $v \Det A$ 
of the Dieudonne determinant $\Det A$ of $A$.
If the extension $v$ is discrete and coincides with  
the Gauss extension (\ref{eqn:Gauss}) on $\KK\langle x_1,x_2,\ldots,x_m \rangle$,
 then one can show by precisely the same argument in \cite{HH_degdet} that $v \Det A$ is given by MVMP.
 Such an extension seems always exist; in this case, $v_p \Det' = v_p \Det$.  
 We verified the existence of an extension with the latter property
 (by adapting Cohn's argument in \cite[Section 4]{Cohn_valuation}). 
 However we could not prove the discreteness. 
 Note that the arguments in this section is independent of the existence issue. 
\end{Rem}

\section*{Acknowledgments}
We thank Kazuo Murota, Satoru Iwata, Satoru Fujishige, 
Yuni Iwamasa for helpful comments, and 
thank Koyo Hayashi for careful reading.
The work was partially supported by JSPS KAKENHI Grant Numbers 25280004, 26330023, 26280004, 17K00029, and JST PRESTO Grant Number JPMJPR192A, Japan.


\begin{thebibliography}{1}
\small
\bibitem{BuildingBook}
P. Abramenko and K. S. Brown, {\em Buildings---Theory and Applications.} 
Springer, New York, 2008.
	
\bibitem{AGLOW}
Z. Allen-Zhu, A. Garg, Y. Li, R. Oliveira, 
and A. Wigderson. 
Operator scaling via geodesically convex optimization, invariant theory and polynomial identity testing. preprint,   2017, the conference version in STOC 2018.


\bibitem{Bacak13}	
M. Ba\v c\'ak, 
The proximal point algorithm in metric spaces. 
{\em Israel Journal of Mathematics} {\bf 194} (2013), 689--701.
%

\bibitem{Bacak14}	
M. Ba\v c\'ak, Computing medians and means in Hadamard spaces. 
{\em SIAM Journal on Optimization} {\bf 24} (2014), 1542--1566.


\bibitem{BacakBook}
M. Ba\v c\'ak, {\em Convex Analysis and Optimization in Hadamard Spaces}. 
De Gruyter, Berlin, 2014.

\bibitem{BHV01}
L. J. Billera, S. P. Holmes, and K. Vogtmann:
Geometry of the space of phylogenetic trees.
{\em Advances in Applied  Mathematics} {\bf 27} (2001), 733--767. 


\bibitem{BM10}
T. Brady and J. McCammond, 
Braids, posets and orthoschemes. 
{\em Algebraic and Geometric Topology} {\bf 10} (2010), 2277--2314.

\bibitem{BrHa}
M. R. Bridson and A. Haefliger, 
{\em Metric Spaces of Non-positive Curvature}. 
Springer-Verlag, Berlin, 1999.


\bibitem{BFGOWW}
P. B\"urgisser, C. Franks, A. Garg, R. Oliveira, 
M. Walter, and A. Wigderson,
Towards a theory of non-commutative optimization: 
geodesic first and second order methods 
for moment maps and polytopes. preprint, 2019,  
the conference version in FOCS 2019.

\bibitem{CCHO}
J. Chalopin, V. Chepoi, H. Hirai, and D. Osajda. 
Weakly modular graphs and nonpositive curvature. 
{\em Memoirs of the AMS}, to appear.


\bibitem{Cohn_valuation}
P. M. Cohn, The construction of valuations of skew fields.
{\em Journal of the Indian Mathematical Society} {\bf 54}
(1989) 1--45. 


\bibitem{Cohn_skew_fields}
P. M. Cohn, {\em Skew fields}. Cambridge University Press, Cambridge, 1995.

\bibitem{DerksenMakam17}
H. Derksen and V. Makam, 
Polynomial degree bounds for matrix semi-invariants.  
{\em Advances in Mathematics} {\bf 310} (2017), 44--63. 




\bibitem{FortinReutenauer04}
M. Fortin and C. Reutenauer, 
Commutative/non-commutative rank of linear
matrices and subspaces of matrices of low rank. 
{\em S\'eminaire Lotharingien de Combinatoire} {\bf 52} (2004), B52f.

\bibitem{FujiBook}
S. Fujishige,
{\it Submodular Functions and Optimization, 2nd Edition.} 
Elsevier, Amsterdam, 2005.

\bibitem{FKMTT14}
S. Fujishige, T Kir\'aly,  K. Makino, K. Takazawa, and S. Tanigawa, 
Minimizing Submodular Functions on Diamonds via Generalized Fractional Matroid Matchings.
EGRES Technical Report (TR-2014-14), (2014).
%

\bibitem{GGOW15}
A. Garg, L. Gurvits, R. Oliveira, and A. Wigderson,
Operator scaling: theory and applications. 
{\em Foundations of Computational Mathematics} (2019).





\bibitem{Gratzer}
G. Gr{\"a}tzer, {\it Lattice Theory: Foundation}. Birkh\"auser, Basel, 2011.


\bibitem{Gurvits04}
L. Gurvits, Classical complexity and quantum entanglement. 
{\em Journal of Computer and System Sciences}  {\bf 69} (2004), 448--484. 



\bibitem{HKS17}
T. Haettel, D. Kielak, and P. Schwer, 
The 6-strand braid group is CAT(0). {\em Geometriae Dedicata} {\bf 182} (2016), 263--286.

\bibitem{HamadaHirai}
M. Hamada and H. Hirai, 
Maximum vanishing subspace problem, CAT(0)-space relaxation, and block-triangularization of partitioned matrix. preprint, 2017. 

\bibitem{Hayashi}
K. Hayashi,
A polynomial time algorithm to compute geodesics in {CAT(0)} 
cubical complexes.
{\em Discrete \& Computational Geometry}, to appear.


\bibitem{HH16DM}
H. Hirai,
Computing DM-decomposition of a partitioned matrix with rank-1 blocks.
{\em Linear Algebra and Its Applications} {\bf 547} (2018), 105--123. 


\bibitem{HH16L-convex}
H. Hirai, L-convexity on graph structures. 
{\em Journal of the Operations Research Society of Japan} {\bf 61} (2018), 71--109.

\bibitem{HH_degdet}
H. Hirai, Computing the degree of determinants via discrete convex optimization on Euclidean buildings.
{\em SIAM Journal on Applied Geometry and Algebra} {\bf 3} (2019), 523--557. 

\bibitem{HI_degdet}
H. Hirai and M. Ikeda,
A cost-scaling algorithm for computing the degree of determinants, preprint, 2020. 

\bibitem{HI_2x2}
H. Hirai and Y. Iwamasa, 
A combinatorial algorithm for computing the rank of a generic partitioned matrix with $2 \times 2$ submatrices.
preprint, 2020,
the conference version in IPCO 2020. 

\bibitem{HrubesWigderson2015}
P. Hrube\v s and A. Wigderson, 
Non-commutative arithmetic circuits with division. 
{\em Theory of Computing} {\bf 11} (2015), 357--393.


\bibitem{IQS15a}
G. Ivanyos, Y. Qiao, and  K. V. Subrahmanyam,
Non-commutative Edmonds' problem and matrix semi-invariants.
{\em Computational Complexity} {\bf 26} (2017) 717--763. 

\bibitem{IQS15b}
G. Ivanyos, Y. Qiao, and  K. V. Subrahmanyam,
Constructive noncommutative rank computation in deterministic polynomial time 
over fields of arbitrary characteristics.
{\em Computational Complexity} {\bf 27} (2018), 561--593.



\bibitem{ItoIwataMurota94}
H. Ito, S. Iwata, and K. Murota, 
Block-triangularizations of partitioned matrices under similarity/equivalence transformations. 
{\em SIAM Journal on Matrix Analysis and Applications}
{\bf 15} (1994), 1226--1255. 

\bibitem{IwataMurota95}
S. Iwata and K. Murota, 
A minimax theorem and a Dulmage-Mendelsohn type decomposition for a class of generic partitioned matrices. 
{\em SIAM Journal on Matrix Analysis and Applications} {\bf 16} (1995), 719--734.
%


\bibitem{Kaliuzhnyi-VerbovetskyiVinnikov2012}
D. S.  Kaliuzhnyi-Verbovetskyi and V. Vinnikov, 
Noncommutative rational functions, their difference-differential calculus and realizations.
{\it Multidimensional Systems and Signal Processing} 
{\bf 23} (2012), 49--77. 


\bibitem{Kuivinen11}
F. Kuivinen,
On the complexity of submodular function minimisation on diamonds.
{\em Discrete Optimization}, {\bf 8} (2011), 459--477.


\bibitem{Lovasz83}
L.  Lov\'asz,
Submodular functions and convexity. 
In A. Bachem, M. Gr\"otschel, and B. Korte (eds.):
{\it Mathematical Programming---The State of the Art} 
(Springer-Verlag, Berlin, 1983), 235--257. 

\bibitem{Lovasz89}
L. Lov\'asz, 
Singular spaces of matrices and their application in combinatorics. 
{\em Boletim da Sociedade Brasileira de Matem\'atica} {\bf 20} (1989), 87--99. 

	
\bibitem{MurotaBook} K. Murota, 
{\em Matrices and Matroids for Systems Analysis.} 
Springer-Verlag, Berlin, 2000.

\bibitem{MurotaDCA}
K. Murota, {\it Discrete Convex Analysis.}
SIAM, Philadelphia, 2004. 

\bibitem{OhtaPalfia15}
S. Ohta and M. P\'alfia,
Discrete-time gradient flows and law of large numbers in
Alexandrov spaces. {\em Calculus of Variations and Partial Differential Equations} {\bf 54}
(2015) 1591--1610.

\bibitem{Oki}
T. Oki,  Computing the maximum degree of minors in skew polynomial matrices. preprint, 2019, the conference version in ICALP 2020.
%

\bibitem{Owen11}
M. Owen, Computing geodesic distances in tree space.
{\em SIAM Journal on Discrete Mathematics} {\bf 25} (2011), 1506--1529.
 
\end{thebibliography}
\end{document}